\documentclass[12pt]{amsart}

\usepackage[T1]{fontenc}
\usepackage[utf8]{inputenc}
\usepackage{amsfonts}
\usepackage{amssymb}
\usepackage{amsmath,amsthm}
\usepackage{mathrsfs}
\usepackage{enumitem}
\usepackage{hyperref}

\makeatletter
\@namedef{subjclassname@2020}{%
  \textup{2020} Mathematics Subject Classification}
\makeatother

\hypersetup{
  colorlinks=true,
  linkcolor=blue,
  citecolor=blue,
  urlcolor=blue
}

\numberwithin{equation}{section}

\newtheorem{theorem}{Theorem}[section]
\newtheorem{lemma}[theorem]{Lemma}
\newtheorem{proposition}[theorem]{Proposition}
\newtheorem{corollary}[theorem]{Corollary}

\theoremstyle{definition}

\newtheorem{example}[theorem]{Example}

\theoremstyle{remark}
\newtheorem{remark}[theorem]{Remark}

\DeclareMathOperator{\Gal}{Gal}
\DeclareMathOperator{\Cl}{Cl}
\DeclareMathOperator{\ord}{ord}
\DeclareMathOperator{\rank}{rank}
\DeclareMathOperator{\length}{length}
\DeclareMathOperator{\charideal}{char}
\DeclareMathOperator{\Ind}{Ind}

\newcommand{\Z}{\mathbb Z}
\newcommand{\Q}{\mathbb Q}
\newcommand{\OO}{\mathcal O}
\newcommand{\mm}{\mathfrak m}

\begin{document}

\title[Ray class groups in the cyclotomic tower]
{Ray class groups in the cyclotomic tower}

\author{Su Hu}
\address{Department of Mathematics,
South China University of Technology,
Guangzhou, Guangdong 510640, China}
\email{mahusu@scut.edu.cn}

\author{Min-Soo Kim}
\address{Department of Mathematics Education,
Kyungnam University,
Changwon, Gyeongnam 51767, Republic of Korea}
\email{mskim@kyungnam.ac.kr}

\date{\today}

\begin{abstract}
Let $p$ be an odd prime and let $\Sigma$ be a nonempty
finite set of rational primes different from $p$.
In the cyclotomic tower $\Q(\mu_{p^{n+1}})$, we determine
the kernels of the natural maps from the minus
$p$-primary ray class groups, with modulus the product
of the primes above $\Sigma$, to the corresponding
ordinary minus $p$-primary class groups.
We give their character components and elementary divisors
at every finite layer, including the layers before
splitting stabilizes.
In the norm inverse limit, we obtain an explicit integral
cyclic decomposition of the local module modulo the
single diagonal global Tate module.
Its coinvariants recover the finite local kernels.
Together with ordinary class-group control, this gives
exact control for the minus ray class groups.

We identify the local Frobenius factors with Bernoulli
regularization factors.
For nonexceptional characters, the classical main
conjecture then yields regularized measure generators
of the ray class characteristic ideals.
The exceptional Teichm\"uller component is treated
without the main conjecture: the single Tate correction
corresponds to the single pole of the $p$-adic zeta function.
At the bottom layer, the local calculation and
Herbrand--Ribet yield characterwise nonvanishing criteria
involving regularized Bernoulli values.
For a single auxiliary prime, the exceptional component
detects the Wieferich condition.
We also deduce Eisenstein congruence consequences
from existing modular-form theorems.
\end{abstract}

\maketitle

\section{Introduction}

Let $p$ be an odd prime and put
\[
K_n=\Q(\mu_{p^{n+1}}),\qquad
K_\infty=\bigcup_{n\geq0}K_n.
\]
Here $\mu_{p^r}$ denotes the group of $p^r$-th roots
of unity.
We fix embeddings of $\overline{\Q}$ into $\mathbb C$
and into $\overline{\Q}_p$.
Algebraic special values are interpreted $p$-adically
through the latter embedding.
We use the valuation $v_p$ on $\overline{\Q}_p$
normalized by $v_p(p)=1$.

We study the change in the negative $p$-primary class
group of $K_n$ when ray congruences are imposed at a fixed
nonempty finite set $\Sigma$ of rational primes different
from $p$.
After $p$-adic completion, the minus part of the global
unit group is $\mu_{p^{n+1}}$.
The local kernel is therefore the minus part of the product
of the residue-field $p$-groups, modulo one diagonal copy
of $\mu_{p^{n+1}}$.
Our aim is to determine this quotient explicitly, both at
each finite layer and in the norm inverse limit.

Hu and Kim \cite{HK16} established an eventual
Iwasawa-type growth formula for the refinement at $2$.
The ray class interpretation and the associated
exact-sequence approach were noted in
\cite[Remark 4.5]{HK16}.
Finat \cite{Finat19} subsequently made the additional
linear growth explicit by evaluating a product of local
character factors.
His calculation in \cite[Section 3]{Finat19} already
treats an arbitrary auxiliary prime.
We determine the underlying local kernels, their character
components, and their elementary divisors at every layer.

The general Iwasawa-theoretic framework for these local
corrections is provided by Greither and Popescu
\cite{GP15} and Kurihara \cite{Kur21}.
Greither and Popescu describe local Tate modules through
Frobenius factors and incorporate the quotient by a
single global copy of $\Z_p(1)$ into their theory of
$p$-adic realizations of $1$-motives.
Kurihara studies Fitting ideals of duals of minus
$T$-ray class groups, using equivariant conjectures
at finite level and a $\mu=0$ hypothesis in the
cyclotomic tower.
For related finite-level results in the Brumer--Stark
setting, see Dasgupta and Kakde \cite{DK23}.

On the analytic side, Knospe and Washington
\cite{KW21} give explicit descriptions of regularized
Bernoulli measures and use them to obtain
Dirichlet-series expansions of $p$-adic $L$-functions.
Their residue-class formula
\cite[Theorem 3.1]{KW21}, together with the classical
integral representation recalled in
\cite[Equation (2)]{KW21}, provides the explicit
analytic input for our measure constructions.
The underlying Bernoulli regularization is classical.
The role of \cite{KW21} here is to make its local
factors and normalization explicit, so that they can
be compared directly with the ray class modules.

Our main contribution is an explicit integral description
of the local quotient in this cyclotomic setting.
For several auxiliary primes, its diagonal Tate submodule
can be absorbed by a single cyclic block, chosen according
to the splitting parameters.
The coinvariants of this quotient recover the finite local
kernels at every layer.
Ordinary class-group control then gives exact control for
the negative ray class groups.
The characteristic-ideal formulas follow from these
calculations and the classical main conjecture.
Using the explicit measure formulas of
Knospe and Washington, we realize the resulting
generators by a symmetric regularization over $\Sigma$.
The exceptional character is treated directly, without
the main conjecture.
None of these analytic inputs is needed for the
finite local structure or the integral Tate quotient.

\subsection*{Notation and finite local corrections}

Let $j\in\Gal(K_\infty/\Q)$ denote complex conjugation
with respect to the fixed embedding into $\mathbb C$.
We use the same symbol for its restriction to each $K_n$.
Thus $j(\zeta)=\zeta^{-1}$ for every root of unity
$\zeta\in K_\infty$.
For a field $F$ and a subgroup $H$ of its automorphism
group, write
\[
F^H=\{x\in F:\sigma(x)=x\text{ for every }\sigma\in H\}
\]
for the fixed field of $H$.
Put
\[
K_n^+=K_n^{\langle j\rangle},
\qquad
K_\infty^+=K_\infty^{\langle j\rangle},
\]
where $\langle j\rangle=\{1,j\}$.
These are the maximal real subfields of $K_n$ and
$K_\infty$, respectively.
More explicitly, if $\zeta_{p^{n+1}}$ is a primitive
$p^{n+1}$-th root of unity, then
\[
K_n^+
=\Q(\zeta_{p^{n+1}}+\zeta_{p^{n+1}}^{-1}),
\qquad
K_\infty^+=\bigcup_{n\geq0}K_n^+.
\]

For a number field $K$ and a finite prime $\mathfrak q$
of $K$, write
\[
k_{\mathfrak q}=\OO_K/\mathfrak q.
\]
For a finite abelian group $M$, write $M[p^\infty]$
for its $p$-primary subgroup.
For a finite abelian $p$-group,
\[
\length_{\Z_p}M=\log_p\#M.
\]

For the fixed set $\Sigma$, define
\[
\mm_{n,\Sigma}
=\prod_{\ell\in\Sigma}
\prod_{\mathfrak l\mid\ell\text{ in }K_n}\mathfrak l,
\qquad
\mm_{n,\Sigma}^+
=\prod_{\ell\in\Sigma}
\prod_{\mathfrak l^+\mid\ell\text{ in }K_n^+}
\mathfrak l^+.
\]
All moduli in this paper are finite and squarefree.
In particular, no positivity condition is imposed at
real places.
Write
\[
A_n=\Cl(K_n)[p^\infty],\qquad
A_{n,\Sigma}
=\Cl_{\mm_{n,\Sigma}}(K_n)[p^\infty].
\]
On a $\Z_p$-module with a $j$-action, put
\[
M^\pm=\frac{1\pm j}{2}M.
\]
Finite abelian $p$-groups, including the $p$-primary
subgroups of residue-field multiplicative groups and
the groups of $p$-power roots of unity, are written
additively when module operations are used.

We use the canonical decomposition
\[
G=\Gal(K_\infty/\Q)=\Delta\times\Gamma,
\qquad
\Delta\simeq(\Z/p\Z)^\times,
\qquad
\Gamma=\Gal(K_\infty/K_0).
\]
All Galois actions on ideals, class groups, and
residue-field groups are the natural actions induced
by field automorphisms.
In particular, $g[\mathfrak a]=[g\mathfrak a]$.
Artin reciprocity is normalized so that prime ideals
away from the conductor map to arithmetic Frobenius.

Let $\omega:\Delta\to\Z_p^\times$ be the Teichm\"uller
character.
All characters of $\Delta$ are taken to have values in
$\Z_p^\times$.
A character $\theta$ is called odd if $\theta(j)=-1$.
For a $\Z_p[\Delta]$-module $M$, put
\[
e_\theta=\frac{1}{p-1}\sum_{\delta\in\Delta}
\theta(\delta)^{-1}\delta,
\qquad
M(\theta)=e_\theta M.
\]
These idempotents are integral, so taking a character
component is exact.

For $a\in\Z_p^\times$, let $\delta_a\in\Delta$
correspond to $a\bmod p$.
We also write
\[
\omega(a)=\omega(\delta_a),
\qquad
\theta(a)=\theta(\delta_a),
\qquad
\langle a\rangle=a/\omega(a)\in1+p\Z_p.
\]
Thus the notation $\omega(\ell)$ or $\theta(\ell)$
for a prime $\ell\ne p$ uses reduction modulo $p$.

Throughout, $\ord_m(a)$ denotes multiplicative order
modulo $m$, for $a$ prime to $m$, not a valuation.
For a prime $\ell\ne p$, define
\begin{equation}\label{eq:parameters}
e_\ell=\ord_p(\ell),\qquad
a_\ell=v_p(\ell^{e_\ell}-1),\qquad
s_\ell=p^{a_\ell-1},
\end{equation}
and put
\begin{equation}\label{eq:b-d}
b_\ell=\frac{p-1}{2\ord_p(\ell^2)}
=\begin{cases}
(p-1)/e_\ell,&e_\ell\text{ even},\\
(p-1)/(2e_\ell),&e_\ell\text{ odd},
\end{cases}
\qquad
d_\ell=b_\ell s_\ell.
\end{equation}
The integer $d_\ell$ is the number of primes above
$\ell$ in $K_\infty^+$; see
Lemma~\ref{lem:decomposition}.
For $n\geq0$, set
\begin{equation}\label{eq:q-m}
q_{n,\ell}=p^{\min(n,a_\ell-1)},
\qquad
m_{n,\ell}=\max(a_\ell,n+1).
\end{equation}
Define
\[
\epsilon_\ell(\theta)=
\begin{cases}
1,&\theta(\delta_\ell)=\omega(\ell),\\
0,&\theta(\delta_\ell)\ne\omega(\ell).
\end{cases}
\]
We call $\theta$ matching for $\ell$ if
$\epsilon_\ell(\theta)=1$.
Finally, put
\[
D_{n,\Sigma}
=\ker\bigl(A_{n,\Sigma}^-\longrightarrow A_n^-\bigr).
\]

Let $h_{n,\Sigma}$ and $h_{n,\Sigma}^+$ be the ray
class numbers for $(K_n,\mm_{n,\Sigma})$ and
$(K_n^+,\mm_{n,\Sigma}^+)$, respectively, and set
\[
h_{n,\Sigma}^-=\frac{h_{n,\Sigma}}{h_{n,\Sigma}^+},
\qquad
h_n^-=\frac{\#\Cl(K_n)}{\#\Cl(K_n^+)}.
\]
Lemma~\ref{lem:relative} relates these quotients to
the minus parts.

\begin{theorem}[Finite local corrections]
\label{thm:finite}
For every $n\geq0$ and every odd character $\theta$,
\begin{equation}\label{eq:finite-character-length}
\length_{\Z_p}D_{n,\Sigma}(\theta)
=\sum_{\ell\in\Sigma}
\epsilon_\ell(\theta)q_{n,\ell}m_{n,\ell}
-(n+1)\mathbf 1_{\theta=\omega},
\end{equation}
where $\mathbf 1_{\theta=\omega}$ is $1$ if
$\theta=\omega$ and $0$ otherwise.
If $\theta\ne\omega$, then
\begin{equation}\label{eq:finite-nonexceptional}
D_{n,\Sigma}(\theta)
\simeq
\bigoplus_{\substack{\ell\in\Sigma\\
\epsilon_\ell(\theta)=1}}
(\Z/p^{m_{n,\ell}}\Z)^{q_{n,\ell}}
\end{equation}
as $\Z_p$-modules.

Choose $\ell_0\in\Sigma$ with $a_{\ell_0}$ minimal,
and abbreviate
$m_0=m_{n,\ell_0}$ and $q_0=q_{n,\ell_0}$.
Then
\begin{equation}\label{eq:finite-exceptional}
\begin{split}
D_{n,\Sigma}(\omega)\simeq{}&
\Z/p^{m_0-n-1}\Z
\oplus(\Z/p^{m_0}\Z)^{q_0-1}\\
&\oplus
\bigoplus_{\ell\in\Sigma\setminus\{\ell_0\}}
(\Z/p^{m_{n,\ell}}\Z)^{q_{n,\ell}}.
\end{split}
\end{equation}
Here and below, $\Z/p^0\Z$ means the zero module,
and an empty direct sum is zero.

Moreover,
\begin{equation}\label{eq:finite-class-number}
v_p(h_{n,\Sigma}^-)-v_p(h_n^-)
=\sum_{\ell\in\Sigma}
b_\ell q_{n,\ell}m_{n,\ell}-(n+1).
\end{equation}
In particular, if
$n+1\geq\max_{\ell\in\Sigma}a_\ell$, then
\begin{equation}\label{eq:stable-class-number}
v_p(h_{n,\Sigma}^-)-v_p(h_n^-)
=(n+1)\Bigl(\sum_{\ell\in\Sigma}d_\ell-1\Bigr).
\end{equation}
\end{theorem}

The theorem determines the local kernel, but not the
extension class of $A_{n,\Sigma}^-$ as an extension
of $A_n^-$ by $D_{n,\Sigma}$.
Before splitting stabilizes, the local blocks have
exponent $p^{a_\ell}$ rather than $p^{n+1}$.
For several auxiliary primes, only one diagonal group
of global roots of unity is removed.
Thus the exceptional component is not the direct sum
of the singleton corrections.
Examples~\ref{ex:5-7} and \ref{ex:multi} illustrate
these points.

\begin{remark}[Finite-level comparison]
\label{rem:finite-comparison}
Theorem~\ref{thm:finite} follows from the ray class
sequence and the decomposition of the auxiliary primes.
It requires neither a main conjecture nor the vanishing
of an Iwasawa $\mu$-invariant.
For $\Sigma=\{2\}$, its stable class-number formula
recovers Finat's comparison; the coefficient $d_\ell$
for an arbitrary auxiliary prime already appears in
\cite[Proposition 3.2]{Finat19}.

The finite-level Fitting-ideal framework of
\cite[Conjecture 3.2 and Corollary 3.7]{Kur21}
concerns the Pontryagin dual of the full minus
$T$-ray class group; see also \cite{DK23} for
related finite-level Brumer--Stark results.
For $K_n/\Q$, the auxiliary prime set denoted by
$T$ in that framework is $\Sigma$ here.
These results and Theorem~\ref{thm:finite} address
different questions: equivariant ideals of the full
ray class group, and the explicit $\Z_p$-structure of
its local kernel.
A Fitting ideal over a group ring does not, in general,
determine that kernel or its elementary divisors.

The torsion-freeness hypothesis on congruence units
in some integral formulations also requires care for
$\Sigma=\{2\}$, since $-1\equiv1\pmod2$.
This prime-to-$p$ torsion causes no difficulty in our
$p$-primary ray class sequence.
\end{remark}

For the bottom-layer applications, put
\[
G_k^{(\ell)}=\ell(1-\ell^k)B_k,
\]
where the Bernoulli numbers are normalized by
\[
\frac{t}{e^t-1}
=\sum_{k\geq0}B_k\frac{t^k}{k!}.
\]
For a nonzero rational number $x$, the notation
$p\mid x$ means $v_p(x)>0$.
An odd prime $p$ is Bernoulli regular if it divides
none of $B_2,B_4,\ldots,B_{p-3}$; for $p=3$ this
condition is vacuous.

\begin{corollary}[A ray class consequence of Herbrand--Ribet]
\label{cor:HR}
Let $p\geq5$, let $\ell\ne p$ be prime, and let
$2\leq k\leq p-3$ be even.
Then
\[
A_{0,\{\ell\}}(\omega^{1-k})\ne0
\quad\Longleftrightarrow\quad
p\mid G_k^{(\ell)}.
\]
Moreover,
\[
\begin{split}
&\length_{\Z_p}A_{0,\{\ell\}}(\omega^{1-k})
-\length_{\Z_p}A_0(\omega^{1-k})\\
&\hspace{30mm}=
\begin{cases}
a_\ell,&\ell^k\equiv1\pmod p,\\
0,&\ell^k\not\equiv1\pmod p.
\end{cases}
\end{split}
\]
\end{corollary}

Here $B_k$ detects the ordinary contribution and
$1-\ell^k$ detects the local one.
The exceptional character $\omega$ instead gives
\[
A_{0,\{\ell\}}(\omega)\simeq\Z/p^{a_\ell-1}\Z.
\]
For this singleton modulus, the exceptional component
is nonzero precisely when
$\ell^{p-1}\equiv1\pmod{p^2}$; see
Proposition~\ref{prop:omega}.

\subsection*{The local inverse limit and exact control}

Choose a topological generator $\gamma$ of $\Gamma$.
Let $\kappa:G\to\Z_p^\times$ be the cyclotomic
character, so that $\kappa|_\Delta=\omega$, and put
\[
u=\kappa(\gamma),\qquad
\Lambda=\Z_p[[\Gamma]]=\Z_p[[T]],
\qquad
\gamma=1+T.
\]
Then $u$ topologically generates $1+p\Z_p$, and
$v_p(u-1)=1$.
For $z\in\Z_p$, the notation $(1+T)^z$ means the
binomial series
\[
(1+T)^z=\sum_{r\geq0}\binom zr T^r\in\Lambda.
\]
Define
\[
L(T)=(1+T)-u,\qquad
R_\ell(T)=(1+T)^{s_\ell}-u^{s_\ell},
\qquad
S_\ell(T)=R_\ell(T)/L(T).
\]
We use the convention $\Lambda/(1)=0$.

Put
\[
\Gamma_n=\Gamma^{p^n},
\qquad
\Omega_n(T)=(1+T)^{p^n}-1.
\]
For a $\Lambda$-module $M$, write
\[
M_{\Gamma_n}=M/\Omega_nM,
\qquad
M[\Omega_n]=\ker(\Omega_n:M\to M).
\]

Let
\[
B_{n,\ell}
=\prod_{\mathfrak l\mid\ell}
k_{\mathfrak l}^{\times}[p^\infty],
\qquad
B_{n,\Sigma}=\bigoplus_{\ell\in\Sigma}B_{n,\ell},
\qquad
Y_\ell=\varprojlim_n B_{n,\ell}.
\]
The transition maps for $Y_\ell$ are residue-field
norms, multiplied over the primes above a given prime
below.
Set
\[
Y_\Sigma=\bigoplus_{\ell\in\Sigma}Y_\ell,
\qquad
Y_\Sigma^-=\frac{1-j}{2}Y_\Sigma,
\]
and
\[
X^-=\varprojlim_n A_n^-,
\qquad
X_\Sigma^-=\varprojlim_n A_{n,\Sigma}^-,
\]
with norm transition maps.
These inverse limits carry their natural compact
topologies, and their Galois actions are continuous.

The Tate module is
\[
\Z_p(1)=\varprojlim_r\mu_{p^r},
\]
where the transition maps are $\zeta\mapsto\zeta^p$.
The group $G$ acts on it through $\kappa$; in
particular,
\[
\Z_p(1)\simeq\Lambda/(L(T))
\]
as a $\Lambda$-module.
Simultaneous reduction induces an embedding
$\Z_p(1)\hookrightarrow Y_\Sigma^-$.
Write
\[
C_\Sigma=Y_\Sigma^-/\Z_p(1).
\]

\begin{theorem}[The Tate quotient and exact control]
\label{thm:correction}
There is a natural exact sequence
\[
0\longrightarrow C_\Sigma
\longrightarrow X_\Sigma^-
\longrightarrow X^-
\longrightarrow0.
\]
If $a_{\ell_0}=\min_{\ell\in\Sigma}a_\ell$, then
\begin{equation}\label{eq:multi-quotient}
\begin{split}
C_\Sigma\simeq{}&
\Lambda/(S_{\ell_0})
\oplus
\bigl(\Lambda/(R_{\ell_0})\bigr)^{b_{\ell_0}-1}\\
&\oplus
\bigoplus_{\ell\in\Sigma\setminus\{\ell_0\}}
\bigl(\Lambda/(R_\ell)\bigr)^{b_\ell}
\end{split}
\end{equation}
as $\Lambda$-modules.
This decomposition is not asserted to be canonical.

For every $n\geq0$, projection induces natural
isomorphisms
\begin{equation}\label{eq:local-descent}
(Y_\ell)_{\Gamma_n}\simeq B_{n,\ell},
\qquad
(C_\Sigma)_{\Gamma_n}\simeq D_{n,\Sigma},
\end{equation}
and
\begin{equation}\label{eq:global-control}
(X^-)_{\Gamma_n}\simeq A_n^-,
\qquad
(X_\Sigma^-)_{\Gamma_n}\simeq A_{n,\Sigma}^-.
\end{equation}
In particular, $C_\Sigma$ is free over $\Z_p$, with
\[
\rank_{\Z_p}C_\Sigma
=\sum_{\ell\in\Sigma}d_\ell-1.
\]
\end{theorem}

The local inverse limit is a twisted permutation module.
On a matching character it is $\Lambda/(R_\ell)$,
and the diagonal Tate embedding is represented in the
$\omega$-component by $(S_\ell)_{\ell\in\Sigma}$.
Since all $s_\ell$ are powers of $p$, choosing a
minimal $a_{\ell_0}$ permits a triangular change of
coordinates that moves the Tate vector into one block.

Local descent follows from residue-field norms and
a computation of $\Lambda/(R_\ell,\Omega_n)$.
Ordinary control follows from the fact that the unique
prime above $p$ is principal and totally ramified in
the tower.
Ferrero--Washington and the absence of finite submodules
give $\Z_p$-freeness; together with ordinary control,
this implies $X^-[\Omega_n]=0$.
The latter vanishing preserves left exactness when
coinvariants are taken in the displayed short exact
sequence, giving ray class control.

\begin{remark}[Infinite-level comparison]
\label{rem:inverse-comparison}
The local module $Y_\ell$ is identified, using stable
residue fields, with the corresponding residue-field
Tate module in \cite{GP15}.
The quotient by one global copy of $\Z_p(1)$ and
its Frobenius factors are already part of that
framework.
Theorem~\ref{thm:correction} gives an explicit integral
cyclic decomposition of this local quotient in the
cyclotomic case and identifies its coinvariants with
the finite local kernels.

Kurihara's Iwasawa-theoretic formula
\cite[Theorem 4.4]{Kur21} concerns the Pontryagin dual
of a direct limit of $T$-ray class groups.
Its $\mu=0$ hypothesis in the cyclotomic setting is
supplied by Ferrero--Washington \cite{FW79}.
Our formulas concern norm inverse limits.
We do not identify these global objects without the
appropriate duality and involution conventions; see
Remark~\ref{rem:general-theory}.
\end{remark}

\subsection*{Analytic factors and applications}

For a finitely generated torsion $\Lambda$-module $M$,
write $\charideal_\Lambda(M)$ for its characteristic
ideal, with $\charideal_\Lambda(0)=\Lambda$.
We use $\mu(M)$ and $\lambda(M)$ for the associated
Iwasawa invariants.

Write arithmetic Frobenius as
\[
\sigma_\ell=\delta_\ell\gamma^{z_\ell},
\qquad
z_\ell\in\Z_p,\qquad
u^{z_\ell}=\langle\ell\rangle.
\]
For an odd character $\theta$, put
\[
\mathcal E_{\ell,\theta}(T)
=1-\ell\theta(\delta_\ell)^{-1}(1+T)^{-z_\ell}.
\]
This is the $\theta$-component of
$1-\ell\sigma_\ell^{-1}$.
It generates $(R_\ell)$ when
$\epsilon_\ell(\theta)=1$, and is a unit otherwise.
Writing
\[
X(\theta)=e_\theta X^-,
\qquad
X_\Sigma(\theta)=e_\theta X_\Sigma^-,
\]
we obtain
\begin{equation}\label{eq:intro-char}
\charideal_\Lambda(X_\Sigma(\theta))
=
\charideal_\Lambda(X(\theta))
\left(
\frac{\prod_{\ell\in\Sigma}\mathcal E_{\ell,\theta}(T)}
{L(T)^{\mathbf 1_{\theta=\omega}}}
\right).
\end{equation}

For odd $\theta\ne\omega$, put
$\chi_\theta=\omega\theta^{-1}$.
The classical power-series realization of the
Kubota--Leopoldt $p$-adic $L$-function gives
$F_\theta\in\Lambda$ satisfying
\[
F_\theta(u^s-1)=L_p(s,\chi_\theta),
\qquad s\in\Z_p.
\]
In our norm inverse-limit convention, the classical
main conjecture gives
\[
\charideal_\Lambda(X(\theta))=(F_\theta).
\]

To realize \eqref{eq:intro-char} by a measure, we use
the regularized Bernoulli measures in the explicit
form described by Knospe and Washington \cite{KW21}.
For an integer $c>1$ prime to $p$, the integral
representation recalled there has regularization factor
\[
1-\chi_\theta(c)\langle c\rangle^{1-s}.
\]
For $c=\ell$, this is exactly
\[
\mathcal E_{\ell,\theta}(u^s-1)
=1-\chi_\theta(\ell)\langle\ell\rangle^{1-s}.
\]
Their residue-class formula fixes the measure
normalization and makes this equality explicit.
Applying the commuting regularization operators for
all $\ell\in\Sigma$ produces a bounded measure,
symmetric in $\Sigma$, whose transform satisfies
\[
\mathscr L_{\Sigma,\theta}
=F_\theta\prod_{\ell\in\Sigma}\mathcal E_{\ell,\theta}.
\]
Thus our measure formula uses the classical
regularization in the explicit formulation of
\cite{KW21}; the additional assertion is that its
factors realize the local ray class correction
computed above.
The construction and the comparison are given in
Section~\ref{sec:analytic}.

For $\theta=\omega$, ordinary control and
$A_0(\omega)=0$ imply $X(\omega)=0$.
The same regularized integral formula allows us to
treat this component directly.
The simple pole of $L_p(s,1)$ cancels exactly one
copy of $L(T)$ in the product of regularization factors.
Consequently, Theorem~\ref{thm:measure-main} gives
\[
\charideal_\Lambda(X_\Sigma(\theta))
=(\mathscr L_{\Sigma,\theta})
\]
for every odd $\theta$.
Exact control then yields
\[
\length_{\Z_p}A_{n,\Sigma}(\theta)
=
\sum_{\zeta^{p^n}=1}
v_p\bigl(\mathscr L_{\Sigma,\theta}(\zeta-1)\bigr),
\]
where the sum is over all $p^n$-th roots of unity
in $\overline{\Q}_p$.

The bottom-layer specialization at $T=0$ differs
from the weight specialization at $T=u^{1-k}-1$.
For odd $\theta\ne\omega$, the former gives
\[
\length_{\Z_p}A_{0,\Sigma}(\theta)
=
v_p\left(
L(0,\theta^{-1})
\prod_{\ell\in\Sigma}(1-\ell\theta^{-1}(\ell))
\right),
\]
whereas, for even $k\geq2$, the latter gives
\[
\mathscr L_{\{\ell\},\omega^{1-k}}(u^{1-k}-1)
=
-(1-p^{k-1})(1-\ell^k)\frac{B_k}{k}.
\]
This distinction explains why
Corollary~\ref{cor:HR} is a divisibility criterion,
rather than a length formula involving
$v_p(G_k^{(\ell)})$.

In Section~\ref{sec:applications}, we compare these
criteria with $\ell$-Genocchi irregularity and the
Wieferich condition, and record consequences of the
Eisenstein congruence theorems of \cite{KKMS23,MS23}.
These are existence results for congruences, not
correspondences between individual ray classes and
individual eigenforms.

\subsection*{Organization}

Section~\ref{sec:ray} establishes the negative ray
class sequence and the relative class-number interpretation.
Section~\ref{sec:finite} computes the finite local
modules and proves Theorem~\ref{thm:finite} and
Corollary~\ref{cor:HR}.
Section~\ref{sec:inverse} determines the local inverse
limit and its Tate quotient, proves exact control,
and derives characteristic ideals and growth formulas.
Section~\ref{sec:analytic} constructs the regularized
measure generators using the explicit formulas of
\cite{KW21}, including the exceptional character.
Section~\ref{sec:applications} treats special values,
modular-form consequences, and vanishing of the local
correction.

\section{The negative ray class sequence}
\label{sec:ray}

For a number field $K$ and a squarefree finite modulus
\[
\mm=\prod_{\mathfrak q\in\mathcal T}\mathfrak q,
\]
let $\Cl_\mm(K)$ denote the ray class group with no
infinite component in the modulus.
Write
\[
E_K=\OO_K^\times,\qquad
E_{K,\mm}
=\{\varepsilon\in E_K:
\varepsilon\equiv1\pmod{\mathfrak q}
\text{ for every }\mathfrak q\mid\mm\}.
\]
The standard ray class sequence is
\begin{equation}\label{eq:ray-sequence}
E_K\longrightarrow
\prod_{\mathfrak q\mid\mm}k_{\mathfrak q}^{\times}
\longrightarrow
\Cl_\mm(K)
\longrightarrow
\Cl(K)
\longrightarrow0.
\end{equation}
Its first map has kernel $E_{K,\mm}$.
This is the finite class-group sequence underlying the
$(S,T)$-refinement with $S=S_\infty$, where
$S_\infty$ denotes the set of infinite places of $K$;
see \cite{Gross88,GP15}.

The sequence below is the finite-level analogue of
the direct-limit sequence \cite[(10), (11)]{GP15}.
It also corresponds to the minus part of the sequence
\cite[(2.4), (2.5)]{Kur21}.
Here we retain the roots of unity explicitly, rather
than impose a torsion-freeness hypothesis on the
congruence units.

\begin{proposition}\label{prop:minus-sequence}
For every $n\geq0$, there is a natural exact sequence
\begin{equation}\label{eq:minus-sequence}
0\longrightarrow\mu_{p^{n+1}}
\longrightarrow B_{n,\Sigma}^-
\longrightarrow A_{n,\Sigma}^-
\longrightarrow A_n^-
\longrightarrow0.
\end{equation}
The first map is simultaneous reduction.
These sequences are compatible with norms.
\end{proposition}

\begin{proof}
Let $\mu(K_n)$ denote the full group of roots of
unity in $K_n$.
For $\varepsilon\in E_{K_n}$, the quotient
$\varepsilon/j(\varepsilon)$ has absolute value $1$
under every complex embedding.
Kronecker's theorem implies that it is a root of unity.
Thus $j$ acts trivially on $E_{K_n}/\mu(K_n)$, and
\[
(E_{K_n}\otimes_{\Z}\Z_p)^-=\mu_{p^{n+1}}.
\]

Reduction of $p$-power roots of unity at a prime away
from $p$ is injective.
Indeed, if a nontrivial root $\zeta$ of order $p^r$
reduced to $1$, then reducing
\[
1+\zeta+\cdots+\zeta^{p^r-1}=0
\]
would force the residue characteristic to divide $p^r$.
Since $\Sigma$ is nonempty, simultaneous reduction is
injective.

Tensor \eqref{eq:ray-sequence} with the flat
$\Z$-module $\Z_p$, and apply the exact idempotent
$(1-j)/2$.
For a finite abelian group $M$, the natural
identification
$M\otimes_{\Z}\Z_p\simeq M[p^\infty]$
gives \eqref{eq:minus-sequence}.

All primes in the moduli are unramified in the tower.
Reduction of a field norm is the product of the
corresponding residue-field norms, and norms preserve
the ray congruences.

The cyclotomic character identifies
$\Gal(K_{n+1}/K_n)$ with
\[
\{1+t\,p^{n+1}\bmod p^{n+2}:0\leq t<p\}.
\]
Consequently, for $\zeta\in\mu_{p^{n+2}}$,
\[
N_{K_{n+1}/K_n}(\zeta)
=\prod_{t=0}^{p-1}\zeta^{1+t\,p^{n+1}}
=\zeta^{\,p+p^{n+2}(p-1)/2}
=\zeta^p.
\]
Here $(p-1)/2$ is an integer because $p$ is odd.
This proves norm compatibility.
\end{proof}

\begin{remark}
Tensoring before applying the minus idempotent is
important.
Taking $p$-power torsion directly in a sequence
containing the infinite unit group is not an exact
operation.
\end{remark}

\begin{lemma}\label{lem:relative}
The number $h_{n,\Sigma}^-$ is a positive integer.
Moreover,
\[
v_p(h_{n,\Sigma}^-)
=\length_{\Z_p}A_{n,\Sigma}^-,
\qquad
v_p(h_n^-)=\length_{\Z_p}A_n^-.
\]
\end{lemma}

\begin{proof}
Let $F=K_n^+$, $K=K_n$, and let $R/F$ be the
ray class field for $\mm_{n,\Sigma}^+$.
The extension $R/F$ is unramified at $p$, whereas
$K/F$ is ramified there.
Since $[K:F]=2$, this gives $R\cap K=F$.
The extension $RK/K$ has conductor dividing
$\mm_{n,\Sigma}$.
Restriction and Artin reciprocity therefore show that
the norm map from the ray class group of $K$ to that
of $F$ is surjective.
Hence $h_{n,\Sigma}^+\mid h_{n,\Sigma}$.

For the $p$-primary assertion, extension of ideals
$i$ and norm $N$ satisfy
\[
Ni=2,\qquad iN=1+j.
\]
Since $p$ is odd, extension identifies the
$p$-primary ray class group of $F$ with
$A_{n,\Sigma}^+$.
Taking orders in
\[
A_{n,\Sigma}=A_{n,\Sigma}^+\oplus A_{n,\Sigma}^-
\]
proves the first valuation formula.
The ordinary formula follows in the same way.
\end{proof}

\section{Finite local modules and character components}
\label{sec:finite}

We compute the local term of
\eqref{eq:minus-sequence} as a $\Z_p[\Delta]$-module
and then remove the diagonal roots of unity.
The computation uses only the decomposition of the
auxiliary primes.

For $\ell\ne p$, let
\[
f_{n,\ell}=\ord_{p^{n+1}}(\ell).
\]

\begin{lemma}\label{lem:orders}
For every $n\geq0$,
\[
f_{n,\ell}
=e_\ell p^{\max(0,n+1-a_\ell)},
\qquad
v_p(\ell^{f_{n,\ell}}-1)=m_{n,\ell}.
\]
\end{lemma}

\begin{proof}
The lifting-the-exponent formula gives
\[
v_p(\ell^{e_\ell t}-1)=a_\ell+v_p(t)
\]
for every positive integer $t$.
Any exponent giving $1$ modulo $p^{n+1}$
is divisible by $e_\ell$.
The least required additional $p$-power is
$p^{\max(0,n+1-a_\ell)}$.
\end{proof}

Put $H_\ell=\langle\delta_\ell\rangle\subseteq\Delta$.
If $\chi:H_\ell\to\Z_p^\times$ and $M$ is a
cyclic $\Z_p$-module, write $M\langle\chi\rangle$
for $M$ equipped with the action $h\cdot x=\chi(h)x$.
This notation specifies an action; it does not denote
a character component.
For an $H_\ell$-module $N$, we use
\[
\Ind_{H_\ell}^{\Delta}N
=\Z_p[\Delta]\otimes_{\Z_p[H_\ell]}N.
\]

\begin{proposition}\label{prop:finite-induced}
As a $\Z_p[\Delta]$-module,
\[
B_{n,\ell}
\simeq
\left(
\Ind_{H_\ell}^{\Delta}
\bigl((\Z/p^{m_{n,\ell}}\Z)
\langle\omega|_{H_\ell}\rangle\bigr)
\right)^{q_{n,\ell}}.
\]
Consequently,
\begin{equation}\label{eq:finite-local-character}
B_{n,\ell}(\theta)
\simeq
\begin{cases}
(\Z/p^{m_{n,\ell}}\Z)^{q_{n,\ell}},
&\epsilon_\ell(\theta)=1,\\
0,&\epsilon_\ell(\theta)=0.
\end{cases}
\end{equation}
\end{proposition}

\begin{proof}
The decomposition group of a prime above $\ell$
has order $f_{n,\ell}$.
Its prime-to-$p$ subgroup is $H_\ell$, of order
$e_\ell$, so its intersection with $\Delta$ is
$H_\ell$.
The number of $\Delta$-orbits on primes above $\ell$ is
\[
\frac{(p-1)p^n/f_{n,\ell}}{(p-1)/e_\ell}
=p^{\min(n,a_\ell-1)}
=q_{n,\ell}.
\]

Each residue-field $p$-group is cyclic of order
$p^{m_{n,\ell}}$.
The action of $H_\ell$ has order prime to $p$.
Its action modulo $p$ is determined by the injected
group $\mu_p$ and is cyclotomic.
The prime-to-$p$ lift of this action is the
Teichm\"uller action $\omega|_{H_\ell}$.
Each orbit therefore gives one copy of the displayed
induced module.

Applying $e_\theta$ gives a nonzero component exactly
when $\theta|_{H_\ell}=\omega|_{H_\ell}$.
When the restrictions agree, that component is one
cyclic group for each orbit.
When they differ, a difference of character values is
a $p$-adic unit, so the component vanishes.
\end{proof}

\begin{proposition}[Elementary divisors]
\label{prop:finite-elementary}
Choose $\ell_0\in\Sigma$ with $a_{\ell_0}$ minimal.
Then \eqref{eq:finite-exceptional} holds.
For a singleton modulus, with
$q=q_{n,\ell}$ and $m=m_{n,\ell}$, this becomes
\[
D_{n,\{\ell\}}(\omega)
\simeq
(\Z/p^m\Z)^{q-1}\oplus\Z/p^{m-n-1}\Z.
\]
Moreover,
\[
D_{n,\{\ell\}}
\simeq
(\Z/p^m\Z)^{b_\ell q-1}
\oplus\Z/p^{m-n-1}\Z
\]
as a $\Z_p$-module.
\end{proposition}

\begin{proof}
Put $r=n+1$, $m_\ell=m_{n,\ell}$, and
$q_\ell=q_{n,\ell}$.
The exceptional part of the negative ray sequence is
\[
0\longrightarrow\mu_{p^r}
\longrightarrow
\bigoplus_{\ell\in\Sigma}
(\Z/p^{m_\ell}\Z)^{q_\ell}
\longrightarrow D_{n,\Sigma}(\omega)
\longrightarrow0.
\]
The projection of a generator of $\mu_{p^r}$ to
each $\ell$-block has order $p^r$, since reduction
at $\ell$ is injective.
After an invertible change of basis in each block,
this projection is $p^{m_\ell-r}$ times its first
basis vector.

Write $m_0=m_{\ell_0}$.
Since $m_0\leq m_\ell$, multiplication by
$p^{m_\ell-m_0}$ defines a homomorphism
\[
\Z/p^{m_0}\Z\longrightarrow\Z/p^{m_\ell}\Z.
\]
Subtracting this multiple of the first coordinate of
the $\ell_0$-block from the first coordinate of each
other block is an automorphism.
It sends the diagonal generator to a vector supported
only in the first coordinate of the $\ell_0$-block.
Quotienting that coordinate replaces
$\Z/p^{m_0}\Z$ by $\Z/p^{m_0-r}\Z$.
This proves \eqref{eq:finite-exceptional}.

For a singleton, the other $b_\ell-1$ nonzero odd
character components remain unchanged, giving the
last formula.
The count of these characters is included below.
\end{proof}

\begin{proof}[Proof of Theorem~\ref{thm:finite}]
The group $\mu_{p^{n+1}}$ has $\Delta$-character
$\omega$.
For odd $\theta\ne\omega$, the negative ray sequence
identifies $D_{n,\Sigma}(\theta)$ with
$B_{n,\Sigma}(\theta)$.
Proposition~\ref{prop:finite-induced} gives
\eqref{eq:finite-nonexceptional}.
For $\theta=\omega$, every local block occurs,
and the roots of unity subtract length $n+1$.
This proves \eqref{eq:finite-character-length}.
The exceptional elementary divisors are given by
Proposition~\ref{prop:finite-elementary}.

There are exactly $b_\ell$ odd characters satisfying
$\epsilon_\ell(\theta)=1$.
Indeed, write $\theta=\omega\eta$.
Then $\eta$ must be trivial on $H_\ell$ and satisfy
$\eta(j)=1$.
If $e_\ell$ is even, then $j\in H_\ell$, and all
$(p-1)/e_\ell$ characters of $\Delta/H_\ell$
qualify.
If $e_\ell$ is odd, exactly half qualify.

Summing over odd characters and using
Lemma~\ref{lem:relative} proves
\eqref{eq:finite-class-number}.
In the stable range,
$q_{n,\ell}=s_\ell$ and $m_{n,\ell}=n+1$, giving
\eqref{eq:stable-class-number}.
\end{proof}

\begin{remark}\label{rem:Finat}
For $\Sigma=\{2\}$, the stable formula is Finat's
comparison.
The formulas above retain the character components
and elementary divisors without a stable-range
restriction.
\end{remark}

\subsection{The bottom layer}

\begin{proof}[Proof of Corollary~\ref{cor:HR}]
Put $\theta=\omega^{1-k}$.
Then $\theta$ is odd and different from $\omega$, and
\[
\theta(\delta_\ell)=\omega(\ell)
\quad\Longleftrightarrow\quad
\ell^k\equiv1\pmod p.
\]
Theorem~\ref{thm:finite} at $n=0$ gives the length
formula and an exact sequence
\[
0\longrightarrow B_{0,\ell}(\theta)
\longrightarrow A_{0,\{\ell\}}(\theta)
\longrightarrow A_0(\theta)
\longrightarrow0.
\]

Herbrand--Ribet \cite{Ribet76,Washington97} gives
\[
A_0(\omega^{1-k})\ne0
\quad\Longleftrightarrow\quad
p\mid B_k.
\]
In the stated range, $B_k$ is $p$-integral.
Thus the ray component is nonzero exactly when
$p\mid B_k$ or $\ell^k\equiv1\pmod p$, which is
equivalent to $p\mid\ell(1-\ell^k)B_k$.
\end{proof}

\begin{proposition}\label{prop:omega}
For every odd prime $p\ne\ell$,
\[
A_{0,\{\ell\}}(\omega)\simeq\Z/p^{a_\ell-1}\Z.
\]
It is nonzero exactly when
\[
\ell^{p-1}\equiv1\pmod{p^2}.
\]
\end{proposition}

\begin{proof}
The classical vanishing $A_0(\omega)=0$ is recalled
in \cite{Washington97,NT3}.
Apply Proposition~\ref{prop:finite-elementary} with
$n=0$.
Since $e_\ell\mid p-1$ and
$p\nmid(p-1)/e_\ell$, lifting the exponent gives
\[
v_p(\ell^{p-1}-1)=a_\ell.
\]
\end{proof}

\begin{remark}
Corollary~\ref{cor:HR} detects the existence of
a class of order $p$ in the specified component.
It does not assert cyclicity of that component,
nor equality of its length with $v_p(G_k^{(\ell)})$.
The exact special-value length formula is given in
Corollary~\ref{cor:bottom-L}.
\end{remark}

\subsection{Examples}

\begin{example}\label{ex:17}
Let $p=17$ and $\ell=2$.
Then
\[
e_2=8,\qquad a_2=1,\qquad b_2=2.
\]
The nonzero odd local components at the bottom layer
have characters $\omega$ and $\omega^9$, and each
component has order $17$.
Since $17$ is Bernoulli regular, $A_0^-=0$.
The roots of unity remove the $\omega$-component, so
\[
A_{0,\{2\}}^-=A_{0,\{2\}}(\omega^9)
\simeq\Z/17\Z.
\]
This agrees with
\[
G_8^{(2)}=2(1-2^8)B_8=17.
\]
\end{example}

\begin{example}\label{ex:5-7}
Let $p=5$ and $\ell=7$.
Since
\[
e_7=4,\qquad 7^4-1=2400,\qquad a_7=2,
\]
we have $b_7=1$ and $s_7=5$.
Only $\omega$ contributes to the local minus part.
Hence
\[
D_{0,\{7\}}\simeq\Z/5\Z,
\qquad
D_{n,\{7\}}\simeq(\Z/5^{n+1}\Z)^4
\quad(n\geq1).
\]
The bottom-layer correction has length $1$,
whereas the stable slope is $d_7-1=4$.
\end{example}

\begin{example}[Two auxiliary primes]\label{ex:multi}
Let $p=5$ and $\Sigma=\{2,7\}$.
For $2$, one has $a_2=b_2=s_2=1$, whereas
$a_7=2$, $b_7=1$, and $s_7=5$.
The only contributing odd character is $\omega$.
Theorem~\ref{thm:finite} gives
\[
D_{0,\{2,7\}}\simeq\Z/25\Z,
\qquad
D_{n,\{2,7\}}\simeq(\Z/5^{n+1}\Z)^5
\quad(n\geq1).
\]
In particular, the bottom-layer group is not the direct
sum of $D_{0,\{2\}}=0$ and $D_{0,\{7\}}\simeq\Z/5\Z$.
There is only one diagonal group of global roots of unity.

At the inverse-limit level,
Theorem~\ref{thm:correction} gives
\[
C_{\{2,7\}}\simeq\Lambda/(R_7),
\]
whose rank over $\Z_5$ is $5$.
\end{example}

\begin{example}[Ordinary and local contributions]
\label{ex:37}
Take $p=37$, $k=32$, and $\ell=31$.
The classical irregular pair $(37,32)$ gives
\[
A_0(\omega^5)\ne0,
\qquad
\omega^{1-32}=\omega^5.
\]
Since $31^2\equiv-1\pmod{37}$, we have $e_{31}=4$.
Also
\[
31^4-1=923520=37\cdot24960,
\qquad
37\nmid24960,
\]
so $a_{31}=1$.
Consequently,
\[
0\longrightarrow\Z/37\Z
\longrightarrow A_{0,\{31\}}(\omega^5)
\longrightarrow A_0(\omega^5)
\longrightarrow0.
\]
The same character receives both an ordinary
class-group contribution and a local contribution.
The displayed sequence does not determine its extension
class.
\end{example}

\section{The local inverse limit, the Tate quotient, and control}
\label{sec:inverse}

We first compute the norm inverse limit of the local
terms and its quotient by the diagonal Tate module.
We then prove local descent and the global control
statements needed to recover the finite ray class groups.

\subsection{Norm limits and local presentations}

\begin{proposition}\label{prop:inverse-sequence}
Taking norm inverse limits in \eqref{eq:minus-sequence}
gives an exact sequence
\begin{equation}\label{eq:inverse-sequence}
0\longrightarrow\Z_p(1)
\longrightarrow Y_\Sigma^-
\longrightarrow X_\Sigma^-
\longrightarrow X^-
\longrightarrow0.
\end{equation}
There is also a natural identification
\[
C_\Sigma\simeq\varprojlim_n D_{n,\Sigma},
\]
where the transition maps are induced by norms.
\end{proposition}

\begin{proof}
Split the finite sequence into
\[
0\longrightarrow\mu_{p^{n+1}}
\longrightarrow B_{n,\Sigma}^-
\longrightarrow D_{n,\Sigma}\longrightarrow0
\]
and
\[
0\longrightarrow D_{n,\Sigma}
\longrightarrow A_{n,\Sigma}^-
\longrightarrow A_n^-\longrightarrow0.
\]
The inverse systems consist of finite groups and
therefore satisfy the Mittag--Leffler condition.
Their inverse limits are exact; see
\cite[Section 3.5]{Weibel94}.
The norm inverse limit of the roots of unity is
$\Z_p(1)$.
The first sequence consequently identifies
$Y_\Sigma^-/\Z_p(1)=C_\Sigma$ with
$\varprojlim_n D_{n,\Sigma}$.
Splicing the resulting sequences proves
\eqref{eq:inverse-sequence}.
\end{proof}

Let $\sigma_\ell\in G$ be arithmetic Frobenius,
normalized by $\kappa(\sigma_\ell)=\ell$, and let
\[
\mathcal D_\ell=\overline{\langle\sigma_\ell\rangle}.
\]
Recall that
\[
\sigma_\ell=\delta_\ell\gamma^{z_\ell},
\qquad
z_\ell=\frac{\log\langle\ell\rangle}{\log u}.
\]

\begin{lemma}\label{lem:decomposition}
One has
\[
v_p(z_\ell)=a_\ell-1,
\qquad
\mathcal D_\ell=H_\ell\times\Gamma^{s_\ell}.
\]
The number of primes above $\ell$ in $K_\infty^+$
is $d_\ell$.
\end{lemma}

\begin{proof}
Since $u$ generates $1+p\Z_p$,
$v_p(\log u)=1$.
Also, $e_\ell$ is prime to $p$ and
$\langle\ell\rangle^{e_\ell}=\ell^{e_\ell}$, so
\[
v_p(\log\langle\ell\rangle)
=v_p(\langle\ell\rangle-1)=a_\ell.
\]
Hence $v_p(z_\ell)=a_\ell-1$.
The prime-to-$p$ and pro-$p$ factors separate in
the closed cyclic subgroup, giving the displayed
decomposition of $\mathcal D_\ell$.

Primes of $K_\infty^+$ above $\ell$ are indexed by
\[
G/\langle\mathcal D_\ell,j\rangle.
\]
Their number is
\[
\frac{(p-1)s_\ell}{|\langle H_\ell,j\rangle|}
=b_\ell s_\ell=d_\ell.
\]
\end{proof}

\begin{proposition}[Local presentation]
\label{prop:local-presentation}
There is an isomorphism of $G$-modules
\begin{equation}\label{eq:permutation-model}
Y_\ell
\simeq
\Z_p[G/\mathcal D_\ell]\otimes_{\Z_p}\Z_p(1)
\end{equation}
with diagonal action.
If $e_c$ denotes the standard basis vector indexed by
$c\in G/\mathcal D_\ell$, the global Tate embedding
is represented by
\begin{equation}\label{eq:diagonal}
z\longmapsto
\sum_{c\in G/\mathcal D_\ell}e_c\otimes z.
\end{equation}
For every character $\theta$,
\begin{equation}\label{eq:local-Lambda}
Y_\ell(\theta)
\simeq
\begin{cases}
\Lambda/(R_\ell),&\epsilon_\ell(\theta)=1,\\
0,&\epsilon_\ell(\theta)=0.
\end{cases}
\end{equation}
\end{proposition}

\begin{proof}
On the cofinal range $n\geq a_\ell-1$, the primes
above $\ell$ no longer split and are indexed by
$G/\mathcal D_\ell$.
Each residue-field $p$-group has order $p^{n+1}$,
so reduction identifies it with $\mu_{p^{n+1}}$.
These identifications are equivariant, and the coordinate
norms correspond to $\zeta\mapsto\zeta^p$.
This proves \eqref{eq:permutation-model}.
Simultaneous reduction gives \eqref{eq:diagonal}.

The $\theta$-component is nonzero precisely when
$\theta\omega^{-1}$ is trivial on $H_\ell$.
In that case its $\Gamma$-module is
\[
\Z_p[\Gamma/\Gamma^{s_\ell}]\otimes_{\Z_p}\Z_p(1).
\]
For the twisted permutation basis
$v_0,\ldots,v_{s_\ell-1}$,
\[
\gamma v_i=u v_{i+1}
\]
with cyclic indices.
Thus $v_0$ generates over $\Lambda$ with relation
\[
(\gamma^{s_\ell}-u^{s_\ell})v_0=0.
\]
Since
\[
R_\ell(T)\equiv T^{s_\ell}\pmod p,
\]
the polynomial $R_\ell$ is distinguished, and
$\Lambda/(R_\ell)$ is free of rank $s_\ell$
over $\Z_p$.
The target has the same rank, so the resulting
surjection is an isomorphism.
This proves \eqref{eq:local-Lambda}.
\end{proof}

The stable-level identifications also identify $Y_\ell$
with the Tate module of the product of the infinite
residue-field multiplicative groups.
This is the cyclotomic characterwise form of the local
structure in \cite{GP15}.

\subsection{The diagonal quotient}

\begin{proposition}\label{prop:quotient-proof}
For odd $\theta\ne\omega$,
\[
C_\Sigma(\theta)
\simeq
\bigoplus_{\substack{\ell\in\Sigma\\
\epsilon_\ell(\theta)=1}}
\Lambda/(R_\ell).
\]
If $a_{\ell_0}$ is minimal, then
\begin{equation}\label{eq:omega-quotient}
C_\Sigma(\omega)
\simeq
\Lambda/(S_{\ell_0})
\oplus
\bigoplus_{\ell\in\Sigma\setminus\{\ell_0\}}
\Lambda/(R_\ell).
\end{equation}
Consequently \eqref{eq:multi-quotient} holds.
\end{proposition}

\begin{proof}
The Tate module has $\Delta$-character $\omega$,
so the other components are unchanged.

In $Y_\ell(\omega)$, use the twisted permutation
basis from Proposition~\ref{prop:local-presentation},
with the invariant vector in the $\Delta/H_\ell$-factor.
The Tate vector is
\[
w_\ell=\sum_{i=0}^{s_\ell-1}v_i.
\]
Since $v_i=u^{-i}\gamma^iv_0$,
\[
w_\ell
=u^{-(s_\ell-1)}
\left(\sum_{i=0}^{s_\ell-1}
\gamma^i u^{s_\ell-1-i}\right)v_0
=u^{-(s_\ell-1)}S_\ell(T)v_0.
\]
After rescaling the cyclic generators by units,
identify
\[
Y_\Sigma(\omega)
\simeq\bigoplus_{\ell\in\Sigma}\Lambda/(R_\ell).
\]
Under this identification, the diagonal Tate submodule
is generated by $(S_\ell)_{\ell\in\Sigma}$.

Put $R_0=R_{\ell_0}$ and $S_0=S_{\ell_0}$.
Since $s_{\ell_0}\mid s_\ell$, there is
$V_\ell\in\Z_p[T]$ such that
\[
R_\ell=V_\ell R_0,\qquad S_\ell=V_\ell S_0.
\]
Multiplication by $V_\ell$ defines a homomorphism
$\Lambda/(R_0)\to\Lambda/(R_\ell)$.
Therefore
\[
(x_0,(x_\ell)_{\ell\ne\ell_0})
\longmapsto
(x_0,(x_\ell-V_\ell x_0)_{\ell\ne\ell_0})
\]
is an automorphism.
It sends the Tate generator to $(S_0,0,\ldots,0)$.
The quotient of its first coordinate is
\[
\Lambda/(R_0,S_0)=\Lambda/(S_0).
\]
This proves \eqref{eq:omega-quotient}.
Adding the other odd character components gives
\eqref{eq:multi-quotient}.
\end{proof}

\begin{remark}
The coordinate changes act only among the
$\omega$-components.
They do not identify distinct $\Delta$-characters.
The character-suppressed decomposition
\eqref{eq:multi-quotient} is a $\Lambda$-module
decomposition.
\end{remark}

Since $s_\ell$ is a power of $p$,
\[
R_\ell(T)\equiv T^{s_\ell}\pmod p,
\qquad
S_\ell(T)\equiv T^{s_\ell-1}\pmod p.
\]
These polynomials are distinguished, with $S_\ell=1$
when $s_\ell=1$.
Thus $C_\Sigma$ is free over $\Z_p$.
Equivalently, the Tate embedding is primitive over
$\Z_p$: its quotient has no $\Z_p$-torsion.

\subsection{Local descent and finite group-ring presentations}

Recall that
\[
\Gamma_n=\Gamma^{p^n},
\qquad
\Omega_n(T)=(1+T)^{p^n}-1.
\]
Under the identification $\gamma=1+T$, the element
$\Omega_n$ corresponds to $\gamma^{p^n}-1$, and
\[
\Lambda/(\Omega_n)
\simeq\Z_p[\Gamma/\Gamma_n].
\]
For a $\Lambda$-module $M$, we write
\[
M_{\Gamma_n}=M/\Omega_nM,
\qquad
M[\Omega_n]=\ker(\Omega_n:M\longrightarrow M).
\]
Here $\Omega_n$ also denotes the endomorphism of $M$
given by multiplication by $\Omega_n(T)$.
In particular, $\Omega_0(T)=T$.

\begin{proposition}\label{prop:descent-proof}
For every $n\geq0$, projection induces
\[
(Y_\ell)_{\Gamma_n}\simeq B_{n,\ell},
\qquad
(C_\Sigma)_{\Gamma_n}\simeq D_{n,\Sigma},
\]
compatibly with $\Delta$.

As modules over $\Z_p[\Gamma/\Gamma_n]$, one has
\begin{equation}\label{eq:finite-group-ring}
D_{n,\Sigma}(\theta)
\simeq
\bigoplus_{\substack{\ell\in\Sigma\\
\epsilon_\ell(\theta)=1}}
\Lambda/(R_\ell,\Omega_n)
\qquad(\theta\ne\omega,\ \theta\text{ odd}),
\end{equation}
and
\begin{equation}\label{eq:finite-group-ring-omega}
D_{n,\Sigma}(\omega)
\simeq
\Lambda/(S_{\ell_0},\Omega_n)
\oplus
\bigoplus_{\substack{\ell\in\Sigma\\\ell\ne\ell_0}}
\Lambda/(R_\ell,\Omega_n),
\end{equation}
where $a_{\ell_0}$ is minimal.
\end{proposition}

\begin{proof}
Finite-field norms are surjective, also on their
$p$-primary parts.
When a prime splits, the product map over primes above
it is surjective.
Hence projection gives a surjection
\[
Y_\ell/\Omega_nY_\ell\longrightarrow B_{n,\ell}.
\]

For a matching character, the source component is
$\Lambda/(R_\ell,\Omega_n)$.
Write $s=s_\ell$.
If $n\geq a_\ell-1$, then $s\mid p^n$, and modulo
$R_\ell$,
\[
\Omega_n\equiv u^{p^n}-1.
\]
Since $v_p(u^{p^n}-1)=n+1$, the underlying group is
\[
(\Z/p^{n+1}\Z)^s.
\]
If $n<a_\ell-1$, then $p^n\mid s$, and modulo
$\Omega_n$,
\[
R_\ell\equiv1-u^s.
\]
Its valuation is $a_\ell$, so the underlying group is
\[
(\Z/p^{a_\ell}\Z)^{p^n}.
\]
These orders agree with
Proposition~\ref{prop:finite-induced}.
Nonmatching components vanish on both sides.
The surjection is therefore an isomorphism.

Taking coinvariants in
\[
\Z_p(1)\longrightarrow Y_\Sigma^-
\longrightarrow C_\Sigma\longrightarrow0
\]
is right exact.
Projection identifies
\[
\Z_p(1)_{\Gamma_n}
=\Z_p(1)/(u^{p^n}-1)\Z_p(1)
\simeq\mu_{p^{n+1}}.
\]
Its map to $B_{n,\Sigma}^-$ is simultaneous reduction,
which is injective.
Hence
\[
(C_\Sigma)_{\Gamma_n}
\simeq B_{n,\Sigma}^-/\mu_{p^{n+1}}
=D_{n,\Sigma}.
\]
The group-ring presentations follow by taking
coinvariants in Proposition~\ref{prop:quotient-proof}.
\end{proof}

\subsection{Noncapitulation, freeness, and global control}

We include the noncapitulation argument needed to
exclude finite submodules.
This also separates that input from ordinary control.
For $m\geq n$, write
\[
i_{n,m}:A_n\longrightarrow A_m,
\qquad
N_{m/n}:A_m\longrightarrow A_n
\]
for extension of ideals and the ideal norm,
respectively.

\begin{lemma}\label{lem:noncapitulation}
For $m\geq n$, extension of ideals induces an
injection
\[
i_{n,m}:A_n^-\longrightarrow A_m^-.
\]
Consequently, $X^-$ has no nonzero finite
$\Lambda$-submodule.
\end{lemma}

\begin{proof}
It suffices to consider $m>n$.
Put $H=\Gal(K_m/K_n)$, and let $\tau$ be the
restriction of $\gamma^{p^n}$ to $K_m$.
Then $\tau$ generates $H$.

If an ideal class $[\mathfrak a]$ capitulates in
$K_m$, choose $\alpha\in K_m^\times$ such that
\[
\mathfrak a\OO_{K_m}=(\alpha).
\]
The map
\[
\sigma\longmapsto\frac{\sigma(\alpha)}{\alpha}
\qquad(\sigma\in H)
\]
is a cocycle with values in $E_{K_m}$.
Its cohomology class is independent of the choices.
If this cocycle is a coboundary, then $\alpha$
differs by a global unit from an element of
$K_n^\times$.
Since extension is injective on fractional ideals,
it follows that $[\mathfrak a]=0$.
Thus the capitulation kernel embeds into
$H^1(H,E_{K_m})$.
This injection is compatible with complex conjugation.

Flatness of $\Z_p$ and exactness of the minus
idempotent allow these operations to be applied to
the finite-group cochain complex.
The negative $p$-primary capitulation kernel therefore
embeds into
\[
H^1\bigl(H,(E_{K_m}\otimes\Z_p)^-\bigr)
=H^1(H,\mu_{p^{m+1}}).
\]
Write $\mu_{p^{m+1}}$ additively.
On this cyclic group, $\tau-1$ acts as multiplication
by $u^{p^n}-1$, whose valuation is $n+1$.
Its image consequently has order $p^{m-n}$.

Let
\[
N_H=1+\tau+\cdots+\tau^{p^{m-n}-1}
\]
be the group-ring norm operator.
By iterating the root-of-unity norm calculation in
Proposition~\ref{prop:minus-sequence}, $N_H$ acts
multiplicatively as $\zeta\mapsto\zeta^{p^{m-n}}$.
Its kernel on $\mu_{p^{m+1}}$ also has order
$p^{m-n}$.
Since $(\tau-1)\mu_{p^{m+1}}$ lies in this kernel,
the two groups are equal.
The cyclic-group formula
\[
H^1(H,\mu_{p^{m+1}})
=\frac{\ker N_H}{(\tau-1)\mu_{p^{m+1}}}
\]
therefore gives zero.
This proves the injectivity of $i_{n,m}$ on minus
parts.

Now let $F$ be a finite $\Lambda$-submodule of $X^-$.
Choose $t\geq1$ and $n_0\geq0$ such that
$p^tF=0$ and $\Gamma_{n_0}$ acts trivially on $F$.
For $x=(x_m)_m\in F$, $n\geq n_0$, and
$m\geq n+t$, norm compatibility and the
extension--norm identity give
\[
i_{n,m}(x_n)
=i_{n,m}(N_{m/n}x_m)
=\sum_{\sigma\in\Gal(K_m/K_n)}\sigma x_m
=p^{m-n}x_m=0.
\]
Injectivity of $i_{n,m}$ implies $x_n=0$.
All lower coordinates are norms of such coordinates,
so $x=0$.
Hence $F=0$.
\end{proof}

\begin{proposition}\label{prop:freeness}
The modules $X^-$, $C_\Sigma$, and $X_\Sigma^-$
are free of finite rank over $\Z_p$.
\end{proposition}

\begin{proof}
The classical module $X^-$ is a finitely generated
torsion $\Lambda$-module.
Ferrero--Washington \cite{FW79} gives $\mu(X^-)=0$,
so it is finitely generated over $\Z_p$.
Its $\Z_p$-torsion is finite and $\Lambda$-stable,
hence zero by Lemma~\ref{lem:noncapitulation}.
Thus $X^-$ is finite free over $\Z_p$.

The freeness of $C_\Sigma$ follows from
Proposition~\ref{prop:quotient-proof}.
The exact sequence
\[
0\longrightarrow C_\Sigma
\longrightarrow X_\Sigma^-
\longrightarrow X^-\longrightarrow0
\]
splits as a sequence of $\Z_p$-modules, proving the
assertion for $X_\Sigma^-$.
No $\Lambda$-linear splitting is asserted.
\end{proof}

\begin{lemma}[Ordinary control]\label{lem:ordinary-control}
For every $n\geq0$, norm projection induces
\[
(X^-)_{\Gamma_n}\simeq A_n^-.
\]
Moreover,
\[
X^-[\Omega_n]=0.
\]
\end{lemma}

\begin{proof}
Let $H_m$ be the Hilbert $p$-class field of $K_m$,
and put
\[
\mathcal M=\bigcup_{m\geq0}H_m.
\]
The inclusion
$H_mK_{m+1}\subseteq H_{m+1}$ shows that the
fields $H_m$ form an increasing sequence.
Moreover,
\[
H_m\cap K_\infty=K_m,
\]
because $H_m/K_m$ is unramified whereas every
nontrivial finite subextension of $K_\infty/K_m$
is ramified at the unique prime above $p$.
Restriction and Artin reciprocity give
\[
\Gal(H_mK_\infty/K_\infty)
\simeq\Gal(H_m/K_m)
\simeq A_m.
\]
Under these identifications, restriction between
successive levels corresponds to the ideal norm.
Since $\mathcal M$ contains $K_\infty$, taking
inverse limits gives
\[
\Gal(\mathcal M/K_\infty)
\simeq\varprojlim_m A_m=:X.
\]

Fix $n$.
For $m\geq n$, the extension $H_m/K_n$ is Galois:
the extension $K_m/K_n$ is Galois, and automorphisms
of $K_m$ preserve its Hilbert $p$-class field.
Choose a compatible system of primes above $p$ in
the fields $H_m$, for $m\geq n$.
The prime above $p$ in $K_m$ is principal, so it
splits completely in $H_m/K_m$.
The corresponding completion of $H_m$ is therefore
the same as that of $K_m$.
Since $K_m/K_n$ is totally ramified at $p$, the
inertia group at the chosen prime in $H_m/K_n$
maps isomorphically onto $\Gal(K_m/K_n)$.
Passing to inverse limits gives an inertia subgroup
$I_n$ that maps isomorphically onto $\Gamma_n$.
Thus inertia supplies a section of
\[
1\longrightarrow X
\longrightarrow\Gal(\mathcal M/K_n)
\longrightarrow\Gamma_n\longrightarrow1.
\]

The conjugation action of this section on $X$ is
the usual $\Gamma_n$-action.
Abelianizing and then killing the image of $I_n$
therefore gives $X/\Omega_nX$.
All inertia subgroups above the unique prime of
$K_n$ over $p$ are conjugate, so their images in
the abelianization coincide.
There are no other ramified finite primes in
$\mathcal M/K_n$, and $K_n$ has no real places.

The resulting quotient corresponds to the maximal
unramified abelian pro-$p$ subextension of
$\mathcal M/K_n$.
This subextension is exactly $H_n$: it is contained
in $H_n$ by the defining property of the Hilbert
$p$-class field, and $H_n$ is already contained in
$\mathcal M$.
Consequently,
\[
X/\Omega_nX\simeq A_n.
\]
Under reciprocity, this is norm projection.
It is therefore compatible with $\Delta$.
Taking minus parts proves the first assertion.

By Proposition~\ref{prop:freeness}, $X^-$ is finite
free over $\Z_p$.
Its quotient by $\Omega_n$ is finite, so
multiplication by $\Omega_n$ becomes invertible after
tensoring with $\Q_p$.
It is consequently injective on $X^-$.
\end{proof}

\begin{remark}
The control isomorphism in
Lemma~\ref{lem:ordinary-control} uses the principal,
unique, totally ramified prime above $p$.
Ferrero--Washington is not needed for that
isomorphism itself.
It is used in Proposition~\ref{prop:freeness}, and
hence in the subsequent proof of
$X^-[\Omega_n]=0$.
\end{remark}

\begin{proposition}[Ray class control]\label{prop:ray-control}
For every $n\geq0$, projection induces a natural
isomorphism
\[
(X_\Sigma^-)_{\Gamma_n}\simeq A_{n,\Sigma}^-.
\]
These isomorphisms identify the coinvariant sequence with
\[
0\longrightarrow D_{n,\Sigma}
\longrightarrow A_{n,\Sigma}^-
\longrightarrow A_n^-\longrightarrow0.
\]
\end{proposition}

\begin{proof}
Apply the snake lemma to multiplication by $\Omega_n$
on
\[
0\longrightarrow C_\Sigma
\longrightarrow X_\Sigma^-
\longrightarrow X^-\longrightarrow0.
\]
Since $X^-[\Omega_n]=0$, the resulting sequence is
\[
0\longrightarrow(C_\Sigma)_{\Gamma_n}
\longrightarrow(X_\Sigma^-)_{\Gamma_n}
\longrightarrow(X^-)_{\Gamma_n}\longrightarrow0.
\]
Projection gives a morphism to the finite ray class
sequence.
The left and right maps are isomorphisms by
Proposition~\ref{prop:descent-proof} and
Lemma~\ref{lem:ordinary-control}.
The middle map is therefore an isomorphism.
\end{proof}

\begin{corollary}\label{cor:omega-tower}
One has $X(\omega)=0$, and hence
$A_n(\omega)=0$ for every $n$.
If $p$ is Bernoulli regular, then $X^-=0$ and
$A_n^-=0$ for every $n$.
\end{corollary}

\begin{proof}
Ordinary control at $n=0$ gives
\[
X(\omega)/TX(\omega)\simeq A_0(\omega)=0.
\]
Nakayama's lemma implies $X(\omega)=0$.
The regular case follows in the same way from
$A_0^-=0$, using Kummer's criterion.
Control then gives the finite-layer assertions.
\end{proof}

\begin{proof}[Proof of Theorem~\ref{thm:correction}]
The exact sequence is obtained from
Proposition~\ref{prop:inverse-sequence}.
The module decomposition is
Proposition~\ref{prop:quotient-proof}.
Local descent and global control are
Proposition~\ref{prop:descent-proof},
Lemma~\ref{lem:ordinary-control}, and
Proposition~\ref{prop:ray-control}.
Taking degrees in \eqref{eq:multi-quotient} gives
\[
\begin{split}
\rank_{\Z_p}C_\Sigma
&=(s_{\ell_0}-1)+(b_{\ell_0}-1)s_{\ell_0}
+\sum_{\ell\ne\ell_0}b_\ell s_\ell\\
&=\sum_{\ell\in\Sigma}d_\ell-1.
\end{split}
\]
\end{proof}

\subsection{Characteristic ideals and growth}

For a finite free $\Z_p$-module $M$ with continuous
$\Gamma$-action, let $\gamma_M$ denote the
automorphism induced by $\gamma$, and set
\[
\Phi_M(T)=\det\bigl((1+T)I-\gamma_M\bigr).
\]
For $M=0$, this determinant is $1$.

Let $d=\rank_{\Z_p}M$, and let $A$ be the matrix
of $\gamma_M$ in a $\Z_p$-basis of $M$.
The continuous image of the pro-$p$ group $\Gamma$
in $\mathrm{GL}_d(\mathbb F_p)$ is a $p$-group.
Thus the reduction of $A$ modulo $p$ is unipotent,
and
\[
\Phi_M(T)\equiv T^d\pmod p.
\]
In particular, $\Phi_M$ is distinguished.
The standard matrix presentation is
\[
0\longrightarrow\Lambda^d
\xrightarrow{(1+T)I_d-A}
\Lambda^d
\longrightarrow M
\longrightarrow0,
\]
where the final map sends the standard basis to the
chosen $\Z_p$-basis of $M$.
Consequently,
\[
\charideal_\Lambda(M)=(\Phi_M(T));
\]
see \cite{GP15} or \cite[Section 10.2]{NT3}.

\begin{corollary}\label{cor:characteristic}
One has
\begin{equation}\label{eq:determinant}
\Phi_{X_\Sigma^-}(T)
=\Phi_{X^-}(T)
\frac{\prod_{\ell\in\Sigma}R_\ell(T)^{b_\ell}}{L(T)}
\end{equation}
and
\begin{equation}\label{eq:characteristic}
\charideal_\Lambda(X_\Sigma^-)
=\charideal_\Lambda(X^-)
\left(
\frac{\prod_{\ell\in\Sigma}R_\ell(T)^{b_\ell}}{L(T)}
\right).
\end{equation}
Consequently,
\[
\mu(X_\Sigma^-)=\mu(X^-)=0,
\qquad
\lambda(X_\Sigma^-)
=\lambda(X^-)+\sum_{\ell\in\Sigma}d_\ell-1.
\]
\end{corollary}

\begin{proof}
Characteristic polynomials multiply in short exact
sequences of finite free $\Z_p$-modules with compatible
$\Gamma$-actions.
The decomposition of $C_\Sigma$ gives
\[
\Phi_{C_\Sigma}(T)
=\frac{\prod_{\ell\in\Sigma}R_\ell(T)^{b_\ell}}{L(T)}.
\]
This is distinguished of degree
$\sum_\ell d_\ell-1$.
The assertions follow.
\end{proof}

\begin{corollary}\label{cor:eventual-growth}
Put
\[
r_\Sigma=\sum_{\ell\in\Sigma}d_\ell-1,
\qquad
\lambda=\lambda(X^-).
\]
There is an integer $c$ such that, for all sufficiently
large $n$,
\[
v_p(h_n^-)=\lambda n+c,
\qquad
v_p(h_{n,\Sigma}^-)
=(\lambda+r_\Sigma)n+(c+r_\Sigma).
\]
\end{corollary}

\begin{proof}
The first formula is the classical Iwasawa growth
formula together with Ferrero--Washington.
The second follows from \eqref{eq:stable-class-number}.
It is also compatible with exact control and
\eqref{eq:determinant}.
\end{proof}

\begin{remark}[Comparison of global objects]
\label{rem:general-theory}
The global objects of \cite{GP15} and
\cite[Theorem 4.4]{Kur21} are built from direct
limits of class groups, through Tate modules or
Pontryagin duals.
The modules used here are norm inverse limits.
We do not identify these global objects.
Such comparisons, including the involution $\#$ of
\cite{Kur21}, require the appropriate duality results
and action conventions.
None of these identifications is needed for the local
decomposition or the control arguments above.
\end{remark}

\section{Frobenius factors and analytic regularization}
\label{sec:analytic}

We identify the local characteristic factors with
regularization factors for Bernoulli measures.
For nonexceptional characters, the classical main
conjecture identifies the ordinary characteristic ideal.
The explicit regularized Bernoulli measures described
by Knospe and Washington \cite{KW21} then realize
its local modification analytically.
For the exceptional character, we use the vanishing
of the ordinary module and the same regularized
integral formula to calculate the polar factor of
the $p$-adic zeta function directly.

\subsection{Local Frobenius factors}

Recall that
\begin{equation}\label{eq:Euler-factor}
\mathcal E_{\ell,\theta}(T)
=1-\ell\theta(\delta_\ell)^{-1}(1+T)^{-z_\ell}
\end{equation}
is the $\theta$-component of
$1-\ell\sigma_\ell^{-1}$, with arithmetic Frobenius.

\begin{proposition}\label{prop:Euler}
For every odd character $\theta$,
\[
(\mathcal E_{\ell,\theta})
=\begin{cases}
(R_\ell),&\epsilon_\ell(\theta)=1,\\
\Lambda,&\epsilon_\ell(\theta)=0.
\end{cases}
\]
In particular,
\[
Y_\ell(\theta)\simeq\Lambda/(\mathcal E_{\ell,\theta}).
\]
\end{proposition}

\begin{proof}
For a nonmatching character, the constant term
$1-\ell\theta(\delta_\ell)^{-1}$ is a unit.

In the matching case, write $z_\ell=s_\ell v$,
where $v\in\Z_p^\times$, and put
\[
W=\frac{(1+T)^{s_\ell}}{u^{s_\ell}}-1\in(p,T).
\]
Then
\[
\mathcal E_{\ell,\theta}
=1-(1+W)^{-v}
=W\bigl(v+O(W)\bigr).
\]
The second factor is a unit and
$R_\ell=u^{s_\ell}W$.
The presentation follows from
Proposition~\ref{prop:local-presentation}.
\end{proof}

\begin{corollary}\label{cor:theta-characteristic}
For every odd $\theta$,
\[
\charideal_\Lambda(X_\Sigma(\theta))
=\charideal_\Lambda(X(\theta))
\left(
\frac{\prod_{\ell\in\Sigma}\mathcal E_{\ell,\theta}(T)}
{L(T)^{\mathbf1_{\theta=\omega}}}
\right).
\]
\end{corollary}

\begin{proof}
Take character components in
\eqref{eq:inverse-sequence}.
The Tate module contributes only for $\theta=\omega$,
with characteristic ideal $(L(T))$.
Use multiplicativity and Proposition~\ref{prop:Euler}.
For $\theta=\omega$, every local factor is divisible
by $L(T)$, so the quotient is integral.
\end{proof}

\subsection{Normalization of the classical main conjecture}

For an even Dirichlet character $\chi$ of conductor
dividing $p$, let $L_p(s,\chi)$ denote the
Kubota--Leopoldt $p$-adic $L$-function with
interpolation convention
\[
L_p(1-k,\chi)
=
\bigl(1-(\chi\omega^{-k})(p)p^{k-1}\bigr)
L(1-k,\chi\omega^{-k})
\qquad(k\geq1).
\]
Here $\chi\omega^{-k}$ is regarded as its associated
primitive Dirichlet character when its value at $p$
and its complex $L$-function are taken.
In particular, for even $k\geq2$,
\[
L_p(1-k,\omega^k)
=(1-p^{k-1})\zeta(1-k).
\]

For odd $\theta\ne\omega$, put
\[
\chi_\theta=\omega\theta^{-1}.
\]
This is a nontrivial even character of conductor $p$.
The classical power-series realization of the
Kubota--Leopoldt $p$-adic $L$-function gives a unique
series $F_\theta(T)\in\Lambda$ satisfying
\begin{equation}\label{eq:analytic-normalization}
F_\theta(u^s-1)=L_p(s,\chi_\theta),
\qquad s\in\Z_p.
\end{equation}
Its existence and integrality belong to the classical
construction of the $p$-adic $L$-function and do not
require the main conjecture.
In the norm inverse-limit convention used here,
the classical main conjecture gives the further identity
\begin{equation}\label{eq:MW}
\charideal_\Lambda(X(\theta))=(F_\theta).
\end{equation}
See \cite{MW84}, \cite[Chapters 13 and 15]{Washington97},
and \cite[Section 10.3]{NT3}.

Suppose instead that the analytic series is written
as $f(V,\chi)$, normalized by
\[
f(u^{1-s}-1,\chi)=L_p(s,\chi).
\]
Then
\begin{equation}\label{eq:change-variable}
F_\theta(T)
=f\bigl(u(1+T)^{-1}-1,\omega\theta^{-1}\bigr).
\end{equation}
Thus the appearance of $u(1+T)^{-1}-1$ is a
change of analytic coordinate.

To specify the inverse-character convention, let
$\psi:\Gamma\to\overline{\Q}_p^\times$ be a
continuous character of finite order.
Extend it to $\Z_p^\times$ by
\[
\psi(a)
=\psi\bigl((\kappa|_\Gamma)^{-1}(\langle a\rangle)\bigr).
\]
This extension is trivial on the Teichm\"uller
subgroup.
The product $\theta^{-1}\psi^{-1}$ is interpreted
as its associated primitive Dirichlet character.
With these conventions,
\begin{equation}\label{eq:finite-order-normalization}
F_\theta(\psi(\gamma)-1)
=L(0,\theta^{-1}\psi^{-1}).
\end{equation}
In particular,
\[
F_\theta(0)=L(0,\theta^{-1}).
\]
A direct verification of
\eqref{eq:finite-order-normalization} from the
regularized measures is given after
Proposition~\ref{prop:KW}.
The value at $T=0$ alone would not determine the
Iwasawa involution.

\begin{corollary}\label{cor:MW-ray}
For every odd $\theta\ne\omega$,
\[
\charideal_\Lambda(X_\Sigma(\theta))
=\left(
F_\theta(T)\prod_{\ell\in\Sigma}
\mathcal E_{\ell,\theta}(T)
\right).
\]
\end{corollary}

\begin{proof}
Combine \eqref{eq:MW} with
Corollary~\ref{cor:theta-characteristic}.
\end{proof}

\subsection{Regularized Bernoulli measures}

For $a\in\Z_p^\times$, put
\[
\alpha(a)=\frac{\log\langle a\rangle}{\log u}\in\Z_p.
\]
Let $E_1$ be the Bernoulli distribution on $\Z_p$,
defined by
\[
E_1(a+p^N\Z_p)=\frac{a}{p^N}-\frac12,
\qquad 0\leq a<p^N.
\]
This distribution is not bounded.
For an integer $c>1$ prime to $p$, let
$\tau_c:\Z_p\to\Z_p$ be the homeomorphism
$\tau_c(x)=cx$.
For a distribution $\mu$, its pushforward is
\[
((\tau_c)_*\mu)(U)=\mu(\tau_c^{-1}U)
\]
for compact open subsets $U$ of $\Z_p$.
Define
\[
\rho_c=1-c\,(\tau_c)_*,
\qquad
E_{1,c}=\rho_cE_1.
\]
Explicitly,
\[
E_{1,c}(a+p^N\Z_p)
=E_1(a+p^N\Z_p)-c\,E_1(c^{-1}a+p^N\Z_p),
\]
where $c^{-1}a+p^N\Z_p$ denotes the inverse-image
coset under $\tau_c$.

The distribution $E_{1,c}$ is a bounded
$\Z_p$-valued measure.
More precisely, the explicit formula in
\cite[Theorem 3.1]{KW21} is
\begin{equation}\label{eq:measure-values}
E_{1,c}(a+p^N\Z_p)
=\frac{c-1}{2}
-\{-a(p^N)^{-1}\}_c,
\end{equation}
where $\{x\}_c$ denotes the representative of a
residue class modulo $c$ in $\{0,\ldots,c-1\}$.
The inverse $(p^N)^{-1}$ is taken modulo $c$.
This formula also follows directly by writing
$cb=a+tp^N$, where $0\leq b<p^N$ and
$0\leq t<c$.

For every odd $\theta$, including $\omega$, set
\begin{equation}\label{eq:measure-series}
\mathscr L_{c,\theta}(T)
=-\int_{\Z_p^\times}
\theta(a)^{-1}(1+T)^{-\alpha(a)}\,dE_{1,c}(a).
\end{equation}
The integral is coefficientwise.
Its coefficients lie in $\Z_p$, since the
binomial-coefficient functions are continuous and
$\Z_p$-valued.
Thus $\mathscr L_{c,\theta}\in\Lambda$.

For every $t\in\overline{\Q}_p$ with $|t|_p<1$,
the series
\[
(1+t)^{-\alpha(a)}
=\sum_{r\geq0}\binom{-\alpha(a)}r t^r
\]
converges uniformly for $a\in\Z_p^\times$,
because all its binomial coefficients are
$p$-adically integral.
The values lie in the finite extension $\Q_p(t)$.
Since $E_{1,c}$ is bounded, specialization at
$T=t$ commutes with integration.
The same argument applies to the other bounded
measures used below.

Write
\[
z_c=\alpha(c),\qquad
\mathcal E_{c,\theta}(T)
=1-c\theta(c)^{-1}(1+T)^{-z_c}.
\]
For prime $c=\ell$, this agrees with
\eqref{eq:Euler-factor}.

\begin{proposition}[Knospe--Washington comparison]
\label{prop:KW}
For odd $\theta\ne\omega$,
\begin{equation}\label{eq:KW-identity}
\mathscr L_{c,\theta}(T)
=\mathcal E_{c,\theta}(T)F_\theta(T).
\end{equation}
Equivalently,
\begin{equation}\label{eq:KW-specialization}
\mathscr L_{c,\theta}(u^s-1)
=\bigl(1-\chi_\theta(c)\langle c\rangle^{1-s}\bigr)
L_p(s,\chi_\theta).
\end{equation}
For $\theta=\omega$, the same specialization holds
with $L_p(s,1)$ when $s\ne1$, and the regularized
product extends across $s=1$.
\end{proposition}

\begin{proof}
The classical integral representation recalled in
\cite[Equation (2)]{KW21} is
\[
-\bigl(1-\chi(c)\langle c\rangle^{1-s}\bigr)L_p(s,\chi)
=\int_{\Z_p^\times}
\chi\omega^{-1}(a)\langle a\rangle^{-s}\,dE_{1,c}(a).
\]
Take $\chi=\omega\theta^{-1}$.
Then
\[
\mathcal E_{c,\theta}(u^s-1)
=1-c\theta(c)^{-1}u^{-sz_c}
=1-\chi_\theta(c)\langle c\rangle^{1-s}.
\]
This proves the specialization formula.

For $\theta\ne\omega$, both sides of
\eqref{eq:KW-identity} are in $\Lambda$.
Their values agree at points with an accumulation
point inside the open $p$-adic unit disc, so the
power series agree.
For $\theta=\omega$, the bounded-measure integral
supplies the extension at $s=1$.
\end{proof}

\begin{remark}[Finite-order specialization]
\label{rem:finite-order-check}
We verify \eqref{eq:finite-order-normalization}.
For $\theta\ne\omega$, choose an integer $c>1$
prime to $p$ such that $\theta(c)\ne\omega(c)$.
Then $\mathcal E_{c,\theta}$ has unit constant term
and is a unit of $\Lambda$.

Let $\psi$ be as in
\eqref{eq:finite-order-normalization}, and put
$\eta=\theta^{-1}\psi^{-1}$.
Since $\langle a\rangle=u^{\alpha(a)}$, the extension
of $\psi$ to $\Z_p^\times$ satisfies
\[
\psi(a)=\psi(\gamma)^{\alpha(a)}.
\]
Consequently, at $T=\psi(\gamma)-1$, the integrand
in \eqref{eq:measure-series} becomes $\eta(a)$.
This is a locally constant nontrivial odd character.

Let $f_\eta$ be its conductor.
Define the generalized Bernoulli number
\[
B_{1,\eta}
=\sum_{a=1}^{f_\eta}
\eta(a)\left(\frac{a}{f_\eta}-\frac12\right)
=\frac1{f_\eta}\sum_{a=1}^{f_\eta}a\,\eta(a),
\]
where $\eta$ is extended by zero to integers not
prime to $f_\eta$.
The second equality uses
$\sum_{a=1}^{f_\eta}\eta(a)=0$.
The standard special-value formula is
\[
L(0,\eta)=-B_{1,\eta}.
\]

For $N$ sufficiently large that $\eta$ factors
through $(\Z/p^N\Z)^\times$, the distribution
relation gives
\[
\sum_{\substack{1\leq a<p^N\\p\nmid a}}
\eta(a)\left(\frac{a}{p^N}-\frac12\right)
=B_{1,\eta}.
\]
This is a finite sum for the locally constant
function $\eta$; it does not require integrating
an arbitrary continuous function against the
unbounded distribution $E_1$.
Regularizing the finite sum therefore gives
\[
-\int_{\Z_p^\times}\eta(a)\,dE_{1,c}(a)
=-(1-c\eta(c))B_{1,\eta}
=(1-c\eta(c))L(0,\eta).
\]
On the other hand,
\[
\mathcal E_{c,\theta}(\psi(\gamma)-1)=1-c\eta(c).
\]
Evaluating Proposition~\ref{prop:KW} and cancelling
this unit proves
\[
F_\theta(\psi(\gamma)-1)=L(0,\theta^{-1}\psi^{-1}).
\]
\end{remark}

For a prime $\ell$, the specialization can also be
computed directly from \eqref{eq:measure-values}:
\[
\mathscr L_{\ell,\theta}(u^s-1)
=-\lim_{N\to\infty}
\sum_{\substack{1\leq a<p^N\\p\nmid a}}
\theta(a)^{-1}\langle a\rangle^{-s}
\left(
\frac{\ell-1}{2}
-\{-a(p^N)^{-1}\}_\ell
\right).
\]

\subsection{The exceptional character}

\begin{proposition}\label{prop:exceptional-analytic}
There is a unit $G(T)\in\Lambda^\times$ such that
\[
L_p(s,1)=\frac{G(u^s-1)}{1-u^{1-s}}
\qquad(s\in\Z_p,\ s\ne1).
\]
For every prime $\ell\ne p$,
\begin{equation}\label{eq:exceptional-measure}
\mathscr L_{\ell,\omega}(T)
=(1+T)G(T)\frac{\mathcal E_{\ell,\omega}(T)}{L(T)}.
\end{equation}
Consequently,
\[
(\mathscr L_{\ell,\omega})=(S_\ell).
\]
These assertions do not use the main conjecture.
\end{proposition}

\begin{proof}
We first prove integrality of the numerator.
Take $c=1+p$.
Then $z_c=\log(1+p)/\log u$ is a $p$-adic unit.
The calculation in Proposition~\ref{prop:Euler},
applied to this integer $c$, shows that
\[
\mathcal E_{c,\omega}(T)=L(T)U_c(T)
\]
for a unit $U_c(T)\in\Lambda^\times$.
Define
\[
G(T)=
\frac{\mathscr L_{c,\omega}(T)}
{(1+T)U_c(T)}\in\Lambda.
\]
By Proposition~\ref{prop:KW}, for $T=u^s-1$ and
$s\ne1$,
\[
G(T)=\frac{L(T)}{1+T}L_p(s,1)
=(1-u^{1-s})L_p(s,1).
\]

Interpolation at $s=0$ gives
\[
L_p(0,1)=L(0,\omega^{-1})=-B_{1,\omega^{-1}},
\]
where
\[
B_{1,\omega^{-1}}
=\frac1p\sum_{a=1}^{p-1}a\,\omega^{-1}(a).
\]
The numerator is congruent to $p-1$ modulo $p$,
so $v_p(L_p(0,1))=-1$.
Since $v_p(1-u)=1$,
\[
G(0)=(1-u)L_p(0,1)\in\Z_p^\times.
\]
Thus $G\in\Lambda^\times$.

For a prime $\ell$, substitute the formula for
$L_p(s,1)$ into Proposition~\ref{prop:KW}.
This gives \eqref{eq:exceptional-measure} after
specialization.
Both sides are integral power series because
$L\mid\mathcal E_{\ell,\omega}$, so the identity
holds in $\Lambda$.
Finally,
\[
(\mathcal E_{\ell,\omega}/L)=(R_\ell/L)=(S_\ell).
\]
\end{proof}

\subsection{A symmetric measure for several auxiliary primes}

Define
\begin{equation}\label{eq:multi-measure}
E_{1,\Sigma}
=\Bigl(\prod_{\ell\in\Sigma}\rho_\ell\Bigr)E_1,
\end{equation}
where the product denotes composition.
The operators commute, since
\[
(\tau_\ell)_*(\tau_{\ell'})_*
=(\tau_{\ell\ell'})_*
=(\tau_{\ell'})_*(\tau_\ell)_*.
\]
Because $\Sigma\ne\varnothing$, one regularization
already produces a bounded $\Z_p$-valued measure,
and every subsequent operator preserves boundedness.
Thus $E_{1,\Sigma}$ is well defined and symmetric
in $\Sigma$.

For every odd $\theta$, put
\[
\mathscr L_{\Sigma,\theta}(T)
=-\int_{\Z_p^\times}
\theta(a)^{-1}(1+T)^{-\alpha(a)}
\,dE_{1,\Sigma}(a).
\]
For a singleton, this agrees with
$\mathscr L_{\ell,\theta}$.

\begin{theorem}[Regularized characteristic generators]
\label{thm:measure-main}
For every odd $\theta$,
\[
\charideal_\Lambda(X_\Sigma(\theta))
=(\mathscr L_{\Sigma,\theta}).
\]
More explicitly,
\begin{equation}\label{eq:multi-analytic}
\mathscr L_{\Sigma,\theta}
=\begin{cases}
F_\theta\displaystyle\prod_{\ell\in\Sigma}
\mathcal E_{\ell,\theta},
&\theta\ne\omega,\\[6pt]
(1+T)G(T)\displaystyle
\frac{\prod_{\ell\in\Sigma}\mathcal E_{\ell,\omega}}{L(T)},
&\theta=\omega.
\end{cases}
\end{equation}
The nonexceptional case uses Mazur--Wiles.
The exceptional case follows from $X(\omega)=0$
and Proposition~\ref{prop:exceptional-analytic}.
\end{theorem}

\begin{proof}
Put
\[
f_{\theta,T}(a)
=\theta(a)^{-1}(1+T)^{-\alpha(a)}.
\]
Multiplicativity gives
\[
f_{\theta,T}(\ell a)
=\theta(\ell)^{-1}(1+T)^{-z_\ell}
f_{\theta,T}(a).
\]
For a bounded measure $\mu$, pushforward therefore gives
\[
\int_{\Z_p^\times}f_{\theta,T}\,d(\rho_\ell\mu)
=
\mathcal E_{\ell,\theta}(T)
\int_{\Z_p^\times}f_{\theta,T}\,d\mu.
\]
All integrals here are coefficientwise.
Starting with any $\ell_1\in\Sigma$, we obtain
\[
\mathscr L_{\Sigma,\theta}
=\mathscr L_{\ell_1,\theta}
\prod_{\ell\in\Sigma\setminus\{\ell_1\}}
\mathcal E_{\ell,\theta}.
\]
Propositions~\ref{prop:KW} and
\ref{prop:exceptional-analytic} give
\eqref{eq:multi-analytic}.

For $\theta\ne\omega$, apply
Corollary~\ref{cor:MW-ray}.
For $\theta=\omega$, use
Corollaries~\ref{cor:omega-tower} and
\ref{cor:theta-characteristic}, together with the fact
that $(1+T)G(T)$ is a unit.
\end{proof}

\begin{remark}[Regularization and the Tate correction]
The factor
\[
1-\chi_\theta(\ell)\langle\ell\rangle^{1-s},
\qquad \chi_\theta=\omega\theta^{-1},
\]
is the regularization factor in the integral
representation recalled in \cite{KW21}.
It is not the factor for omission of an ordinary
Euler factor, which in this normalization is
\[
1-\theta^{-1}(\ell)\langle\ell\rangle^{-s}.
\]
The additional multiplier $\ell$ in the ray class
factor is essential.

For $\theta=\omega$, every regularization factor
vanishes at $s=1$, whereas $L_p(s,1)$ has one pole.
Thus precisely one factor $L(T)$ is removed.
This matches the single diagonal global Tate module,
rather than one Tate module for each auxiliary prime.
\end{remark}

When $\ell=2$, \eqref{eq:measure-values} gives
\[
E_{1,2}(a+p^N\Z_p)=\frac{(-1)^a}{2}.
\]
The alternating, or fermionic, measure is therefore
twice $E_{1,2}$.
This specifies the normalization relating these factors
to the Euler-type constructions in \cite{KH12};
see also \cite[Section 4]{KW21}.

\subsection{Specialization and finite lengths}

\begin{corollary}\label{cor:finite-analytic-length}
For every odd $\theta$ and every $n\geq0$,
\[
\length_{\Z_p}A_{n,\Sigma}(\theta)
=\sum_{\zeta^{p^n}=1}
v_p\bigl(\mathscr L_{\Sigma,\theta}(\zeta-1)\bigr),
\]
where the sum is over all $p^n$-th roots of unity
in $\overline{\Q}_p$.
All values occurring in the sum are nonzero.
\end{corollary}

\begin{proof}
Put $M=X_\Sigma(\theta)$.
By Proposition~\ref{prop:freeness}, $M$ is a finite
free $\Z_p$-module.
Let $\gamma_M$ denote the automorphism of $M$
induced by $\gamma$.
Recall that
\[
\Omega_n(T)=(1+T)^{p^n}-1.
\]
Since $\gamma=1+T$, multiplication by $\Omega_n$
on $M$ is the endomorphism $\gamma_M^{p^n}-I$.
Ray class control, followed by taking the
$\theta$-component, gives
\[
M/\Omega_nM
\simeq A_{n,\Sigma}(\theta).
\]
In particular, this quotient is finite.
Thus $\gamma_M^{p^n}-I$ becomes invertible on
$M\otimes_{\Z_p}\Q_p$ and is injective on $M$.
The elementary divisor theorem over $\Z_p$ therefore
gives
\[
\length_{\Z_p}A_{n,\Sigma}(\theta)
=
v_p\!\left(\det(\gamma_M^{p^n}-I)\right).
\]

Recall also that
\[
\Phi_M(T)=\det\bigl((1+T)I-\gamma_M\bigr).
\]
Over a finite extension of $\Q_p$ containing all
$p^n$-th roots of unity, the factorization
\[
X^{p^n}-1=\prod_{\zeta^{p^n}=1}(X-\zeta)
\]
gives, up to a sign,
\[
\det(\gamma_M^{p^n}-I)
=
\pm\prod_{\zeta^{p^n}=1}
\det(\zeta I-\gamma_M)
=
\pm\prod_{\zeta^{p^n}=1}\Phi_M(\zeta-1).
\]
The determinant on the left is nonzero, so every
factor on the right is nonzero.
Taking $p$-adic valuations yields
\[
\length_{\Z_p}A_{n,\Sigma}(\theta)
=
\sum_{\zeta^{p^n}=1}v_p\bigl(\Phi_M(\zeta-1)\bigr).
\]

The determinant presentation of $M$ and
Theorem~\ref{thm:measure-main} identify the same
characteristic ideal:
\[
(\Phi_M)
=\charideal_\Lambda(M)
=(\mathscr L_{\Sigma,\theta}).
\]
Consequently,
\[
\mathscr L_{\Sigma,\theta}(T)=U(T)\Phi_M(T)
\]
for some $U(T)\in\Lambda^\times$.
For every $p$-power root of unity $\zeta$, one has
$|\zeta-1|_p<1$.
Since $U(0)\in\Z_p^\times$ and all coefficients
of $U$ are integral, $U(\zeta-1)$ is a $p$-adic
unit.
Hence
\[
v_p\bigl(\mathscr L_{\Sigma,\theta}(\zeta-1)\bigr)
=
v_p\bigl(\Phi_M(\zeta-1)\bigr).
\]
Substitution proves the asserted length formula
and also the nonvanishing of all values occurring
in the sum.
\end{proof}

\begin{corollary}\label{cor:bottom-L}
For every odd $\theta\ne\omega$,
\[
\length_{\Z_p}A_0(\theta)
=v_p(L(0,\theta^{-1}))
\]
and
\[
\length_{\Z_p}A_{0,\Sigma}(\theta)
=v_p\left(
L(0,\theta^{-1})
\prod_{\ell\in\Sigma}(1-\ell\theta^{-1}(\ell))
\right).
\]
For the exceptional character,
\[
\length_{\Z_p}A_{0,\Sigma}(\omega)
=\sum_{\ell\in\Sigma}a_\ell-1.
\]
\end{corollary}

\begin{proof}
Ordinary control, \eqref{eq:MW}, and the determinant
argument above at $n=0$ give the first formula.
For the second, specialize
Theorem~\ref{thm:measure-main} and
Corollary~\ref{cor:finite-analytic-length} at $n=0$.
For $\omega$, use \eqref{eq:multi-analytic},
$v_p(1-u)=1$, and
\[
v_p(\mathcal E_{\ell,\omega}(0))
=v_p(1-\langle\ell\rangle)=a_\ell.
\]
\end{proof}

\begin{remark}
These are equalities of lengths, not assertions that
the full class-group components are cyclic.
For $\theta=\omega^{1-k}$, the bottom-layer length
is obtained at $s=0$, whereas $G_k^{(\ell)}$
arises at $s=1-k$.
This explains the distinction between
Corollaries~\ref{cor:HR} and \ref{cor:bottom-L}.

For several auxiliary primes, the exceptional
bottom-layer component is nonzero whenever
$|\Sigma|\geq2$, since
$\sum_{\ell\in\Sigma}a_\ell-1\geq|\Sigma|-1$.
Thus its nonvanishing is equivalent to a Wieferich
condition only in the singleton case.
\end{remark}

\section{Special values and modular-form consequences}
\label{sec:applications}

We evaluate the analytic generators at negative
integers and compare the resulting criteria with
$\ell$-Genocchi irregularity, Wieferich primes, and
Eisenstein congruences of prime level.

\subsection{Normalized special values}

Put
\[
H_k^{(\ell)}=(1-\ell^k)\frac{B_k}{k},
\qquad
G_k^{(\ell)}=\ell kH_k^{(\ell)}.
\]
For even $k\geq2$, take
$\theta=\omega^{1-k}$, allowing $\theta=\omega$.
Interpolation and regularization give
\begin{equation}\label{eq:weight-value}
\mathscr L_{\ell,\omega^{1-k}}(u^{1-k}-1)
=-(1-p^{k-1})H_k^{(\ell)}.
\end{equation}
Indeed,
\[
L_p(1-k,\omega^k)
=(1-p^{k-1})\zeta(1-k)
=-(1-p^{k-1})\frac{B_k}{k},
\]
and the regularization factor is $1-\ell^k$.

An odd prime $p\ne\ell$ is called
$\ell$-Genocchi regular if none of
\[
G_2^{(\ell)},G_4^{(\ell)},\ldots,G_{p-3}^{(\ell)}
\]
is divisible by $p$.
For $p=3$, this condition is vacuous.
Define the Wieferich set
\[
\mathcal W_\ell
=\{p>2:\ p\ne\ell,\ 
\ell^{p-1}\equiv1\pmod{p^2}\}.
\]
We reserve $W_\ell$ without a calligraphic letter
for the Atkin--Lehner operator below.

The regularity criterion of
\cite[Proposition 4]{MS23} is
\[
\begin{split}
p\text{ is $\ell$-Genocchi regular}
\quad\Longleftrightarrow\quad&
p\text{ is Bernoulli regular}\\
&\text{and }\ord_p(\ell^2)=(p-1)/2.
\end{split}
\]
For $\ell=2$, this criterion and its connection with
refined class numbers were established in \cite{HKMS19}.

\begin{corollary}\label{cor:sequence}
For every odd prime $p\ne\ell$,
\[
v_p(h_{0,\{\ell\}}^-)
=v_p(h_0^-)+a_\ell b_\ell-1.
\]
The following conditions are equivalent:
\begin{enumerate}[label=\textup{(\roman*)}]
\item $p\mid h_{0,\{\ell\}}^-$;
\item $p$ is $\ell$-Genocchi irregular or
$p\in\mathcal W_\ell$;
\item $p\mid H_k^{(\ell)}$ for some even $k\geq2$;
\item $p\mid H_k^{(\ell)}$ for some even $k\geq4$.
\end{enumerate}
\end{corollary}

\begin{proof}
The valuation formula is
\eqref{eq:finite-class-number} at $n=0$.
Since $a_\ell,b_\ell\geq1$,
\[
a_\ell b_\ell-1>0
\quad\Longleftrightarrow\quad
a_\ell>1\text{ or }b_\ell>1.
\]
Now $a_\ell>1$ is the Wieferich condition, and
$b_\ell=1$ is equivalent to
$\ord_p(\ell^2)=(p-1)/2$.
Kummer's criterion and the displayed regularity
criterion give the equivalence of (i) and (ii).

For the remaining equivalences, suppose first that
$p-1\nmid k$.
Let $r$ be the even integer with
$2\leq r\leq p-3$ and $r\equiv k\pmod{p-1}$.
The Kummer congruences give
\[
(1-p^{k-1})\frac{B_k}{k}
\equiv
(1-p^{r-1})\frac{B_r}{r}\pmod p.
\]
The two sides are $p$-integral, and both Euler
factors are congruent to $1$ modulo $p$.
Also $\ell^k\equiv\ell^r\pmod p$.
Thus divisibility in these classes is detected by
$\ell$-Genocchi irregularity.

If $p-1\mid k$, von Staudt--Clausen and
lifting the exponent give
\[
v_p(H_k^{(\ell)})
=a_\ell+v_p(k)-1-v_p(k)
=a_\ell-1.
\]
This detects $\mathcal W_\ell$.
These arguments also cover $p=3$, when all even
weights lie in the exceptional class.

Finally, a nonexceptional weight $2$ can be replaced
by $2+(p-1)$.
In the exceptional class, choose an even multiple of
$p-1$ at least $4$.
This proves all the equivalences; compare
\cite[Lemma 4 and Proposition 3]{MS23}.
\end{proof}

\begin{remark}
The normalization by $k$ matters when the index varies.
The identity $G_k^{(\ell)}=\ell kH_k^{(\ell)}$
introduces additional prime divisors into the full
Genocchi sequence; see \cite[Proposition 2]{MS23}.
\end{remark}

\subsection{Eisenstein congruences}

Let $z$ lie in the complex upper half-plane, put
$q=e^{2\pi iz}$, and write
\[
\sigma_{k-1}(n)=\sum_{d\mid n}d^{k-1}.
\]
All spaces $S_k(\Gamma_0(\ell))$ below have trivial
nebentypus.
For even $k\geq4$, normalize
\[
E_k(z)
=-\frac{B_k}{2k}
+\sum_{n\geq1}\sigma_{k-1}(n)q^n,
\]
and define
\[
E_{k,\ell}^{\varepsilon}(z)
=E_k(z)+\varepsilon\ell^{k/2}E_k(\ell z),
\qquad \varepsilon\in\{\pm1\}.
\]
Its constant term is
\[
-\frac12H_k^{\varepsilon,(\ell)},
\qquad
H_k^{\varepsilon,(\ell)}
=(1+\varepsilon\ell^{k/2})\frac{B_k}{k}.
\]

We normalize the Atkin--Lehner operator by
\[
(f|_kW_\ell)(z)
=\ell^{-k/2}z^{-k}f\!\left(-\frac1{\ell z}\right).
\]
For even $k$, this is an involution.
With this normalization,
\[
E_k|_kW_\ell=\ell^{k/2}E_k(\ell z),
\qquad
E_{k,\ell}^{\varepsilon}|_kW_\ell
=\varepsilon E_{k,\ell}^{\varepsilon}.
\]
Let $S_k^\varepsilon(\ell)$ be the
$\varepsilon$-eigenspace in $S_k(\Gamma_0(\ell))$.

A congruence modulo a prime $\mathfrak P$ in a
coefficient field means coefficientwise congruence
in the corresponding local integer ring.
A normalized form has first Fourier coefficient $1$.
The following two inputs are distinct: the first gives
Hecke eigenfunctions away from the level, and the
second gives newforms.

\begin{theorem}[Existing Eisenstein--cuspidal congruences]
\label{thm:external-eisenstein}
Let $k\geq4$ be even, and let $p\geq5$ and
$\ell\ne p$ be primes.
\begin{enumerate}[label=\textup{(\roman*)}]
\item If $\varepsilon\in\{\pm1\}$ and
$p\mid H_k^{\varepsilon,(\ell)}$, then there are
a normalized cusp form $f\in S_k^\varepsilon(\ell)$,
which is an eigenfunction of every Hecke operator
$T_q$ for prime $q\ne\ell$, and a prime
$\mathfrak P$ above $p$ in its coefficient field,
such that
\[
f\equiv E_{k,\ell}^{\varepsilon}\pmod{\mathfrak P}.
\]
\item If $p\mid H_k^{(\ell)}$, the conclusion in
\textup{(i)} holds for at least one sign.
\end{enumerate}
The congruence includes the constant term.
There is no hypothesis $p>k+1$, no exclusion
$\ell\equiv-1\pmod p$, and no assumption
$p\nmid\ell^2-1$.
\end{theorem}

\begin{proof}
Part (i) is the $k\geq4$, prime-to-level case of
\cite[Corollary 4.2]{KKMS23}, also stated as
\cite[Theorem 2(ii)]{MS23}.
Their divisibility convention is divisibility of the
reduced numerator, as used here.
The factor relating their hypothesis to ours is
the $p$-adic unit $2\varepsilon$.

Part (ii) is \cite[Theorem 2(i)]{MS23}.
It also follows from (i).
Put $x=\ell^{k/2}$.
Since $p$ is odd, at least one of $1-x$ and $1+x$
is a $p$-adic unit.
If
\[
v_p\left((1-x^2)\frac{B_k}{k}\right)>0,
\]
then at least one of
\[
v_p\left((1-x)\frac{B_k}{k}\right),
\qquad
v_p\left((1+x)\frac{B_k}{k}\right)
\]
is positive.
This argument allows $B_k/k$ to have negative
valuation.
\end{proof}

\begin{theorem}[Newform congruences of \cite{KKMS23}]
\label{thm:external-newform}
Let $k\geq4$ be even, let $p\geq5$ and
$\ell\ne p$ be primes, and let $\varepsilon\in\{\pm1\}$.
Suppose
\[
p\mid(\varepsilon+\ell^{k/2})\frac{B_k}{2k},
\qquad
k\not\equiv0\pmod{p-1},
\qquad
p\nmid\frac{B_k}{2k}.
\]
Then there are a normalized newform
$f\in S_k^\varepsilon(\ell)$ and a prime
$\mathfrak P$ above $p$ such that
\[
f\equiv E_{k,\ell}^{\varepsilon}\pmod{\mathfrak P}.
\]
\end{theorem}

\begin{proof}
This is \cite[Theorem 1.3]{KKMS23}, with their
level prime denoted here by $\ell$ and their
congruence prime denoted here by $p$.
That theorem also permits $k=2$ with
$\varepsilon=-1$; only the $k\geq4$ case is used here.
The additional restrictions in their Theorem 1.2
are not being imposed or used here.
\end{proof}

\begin{remark}[Exceptional weights and integrality]
If $p-1\mid k$, then
\[
v_p(B_k/k)=-1-v_p(k).
\]
Nevertheless, if $a_\ell>1$, one sign
$\varepsilon$ makes
$H_k^{\varepsilon,(\ell)}$ divisible by $p$.
Choose the sign such that
$\ell^{k/2}\equiv-\varepsilon\pmod p$.
The other factor of $1-\ell^k$ is a unit, so
\[
v_p(H_k^{\varepsilon,(\ell)})=a_\ell-1.
\]
The positive Fourier coefficients of
$E_{k,\ell}^{\varepsilon}$ are integral, and its
constant term is then $p$-integral and vanishes
modulo $p$.
Thus Theorem~\ref{thm:external-eisenstein} applies
in this Wieferich case.
Theorem~\ref{thm:external-newform}, which excludes
$k\equiv0\pmod{p-1}$, is not being applied.
\end{remark}

\subsection{Ray class interpretations}

\begin{corollary}\label{cor:modular}
Let $p\geq5$ and $\ell\ne p$ be primes.
Then $p\mid h_{0,\{\ell\}}^-$ if and only if,
for some even weight $k\geq4$ and some sign
$\varepsilon$, there is a congruence of the kind in
Theorem~\ref{thm:external-eisenstein}.

For a fixed even weight $4\leq k\leq p-3$,
\[
A_{0,\{\ell\}}(\omega^{1-k})\ne0
\]
if and only if such a congruence exists in weight $k$
for at least one sign.
\end{corollary}

\begin{proof}
The forward implications follow from
Corollaries~\ref{cor:sequence} and \ref{cor:HR},
respectively, and
Theorem~\ref{thm:external-eisenstein}.
In the fixed-weight range, $\ell k$ is a
$p$-adic unit.

Conversely, comparing constant terms gives
$p\mid H_k^{\varepsilon,(\ell)}$.
Since
\[
H_k^{(\ell)}
=(1-\varepsilon\ell^{k/2})H_k^{\varepsilon,(\ell)}
\]
and the first factor is an integer, one obtains
$p\mid H_k^{(\ell)}$.
Apply Corollary~\ref{cor:sequence}, or
Corollary~\ref{cor:HR} in the fixed-weight range.
\end{proof}

\begin{corollary}\label{cor:local-newform}
Let $p\geq5$ and $\ell\ne p$ be primes, and let
$4\leq k\leq p-3$ be even.
Suppose
\[
p\nmid B_k,\qquad
A_{0,\{\ell\}}(\omega^{1-k})\ne0.
\]
Then
\[
A_{0,\{\ell\}}(\omega^{1-k})
\simeq\Z/p^{a_\ell}\Z,
\]
and there is a normalized newform
$f\in S_k^\varepsilon(\ell)$ satisfying
\[
f\equiv E_{k,\ell}^{\varepsilon}\pmod{\mathfrak P}
\]
for a prime $\mathfrak P$ above $p$, where
\[
\ell^{k/2}\equiv-\varepsilon\pmod p.
\]
\end{corollary}

\begin{proof}
Herbrand--Ribet gives $A_0(\omega^{1-k})=0$.
The finite local sequence identifies the ray component
with $\Z/p^{a_\ell}\Z$ and gives
$\ell^k\equiv1\pmod p$.
Since $p$ is odd, the sign is uniquely determined.

The fixed-weight assumptions imply
\[
k\not\equiv0\pmod{p-1},
\qquad
v_p(B_k/(2k))=0.
\]
The chosen sign gives
\[
p\mid(\varepsilon+\ell^{k/2})B_k/(2k).
\]
All hypotheses of
Theorem~\ref{thm:external-newform} are satisfied.
\end{proof}

\begin{example}[An exceptional-weight congruence]
\label{ex:wieferich-modular}
Return to $p=5$ and $\ell=7$.
Proposition~\ref{prop:omega} gives
\[
A_{0,\{7\}}(\omega)\simeq\Z/5\Z.
\]
At the exceptional weight $k=4=p-1$,
\[
H_4^{(7)}=20,
\qquad
H_4^{+,(7)}=-\frac5{12}.
\]
Theorem~\ref{thm:external-eisenstein} therefore gives
a normalized cusp form in $S_4^+(7)$, an eigenfunction
away from $7$, congruent to $E_{4,7}^+$ modulo a
prime above $5$.
This uses \cite[Corollary 4.2]{KKMS23}, not their
Theorem 1.3.
The nonintegrality of $B_4/4=-1/120$ causes no
problem after regularization.
\end{example}

\begin{remark}
Complex conjugation on class groups and the
Atkin--Lehner operator on modular forms are different
involutions.
The factorization of $1-\ell^k$ does not provide an
Atkin--Lehner decomposition of the ray class module.
The results above assert existence of congruences,
not a correspondence between individual ray classes
and individual eigenforms.
\end{remark}

\subsection{Vanishing of the local correction}

For a singleton modulus,
\[
C_{\{\ell\}}
\simeq
\bigl(\Lambda/(R_\ell)\bigr)^{b_\ell-1}
\oplus\Lambda/(S_\ell).
\]
Consequently,
\[
C_{\{\ell\}}=0
\quad\Longleftrightarrow\quad
a_\ell=b_\ell=1.
\]
Equivalently,
\[
\ord_p(\ell^2)=(p-1)/2,
\qquad
\ell^{p-1}\not\equiv1\pmod{p^2}.
\]

To formulate the density issue, now fix $\ell$ and
let $p$ vary.
The dependence of $C_{\{\ell\}}$ on $p$ is
understood in this paragraph.
Put
\[
\mathcal P_\ell
=\{p>2:\ p\ne\ell,\ \ord_p(\ell^2)=(p-1)/2\}.
\]
The set of primes for which $C_{\{\ell\}}=0$ is
exactly
\[
\mathcal P_\ell\setminus\mathcal W_\ell.
\]
The near-primitive-root density results in
\cite{HKMS19,MS23} concern $\mathcal P_\ell$ and
must be used with the hypotheses stated there,
including the relevant generalized Riemann hypothesis
assumptions for the Artin-type density assertions.
No unconditional density assertion is made here.

To transfer such a density to the vanishing set,
it suffices to know that
\[
\#(\mathcal P_\ell\cap\mathcal W_\ell\cap[1,x])
=o(x/\log x).
\]
Writing
\[
\mathcal W_\ell(x)
=\#\{p\leq x:p\in\mathcal W_\ell\},
\]
the stronger assertion
$\mathcal W_\ell(x)=o(x/\log x)$ would suffice
as well.
It is not supplied by the near-primitive-root results;
see \cite[Section 3.4]{MS23}.

Returning to a fixed $p$ and a general nonempty
$\Sigma$, one has
\[
\rank_{\Z_p}C_\Sigma
=\sum_{\ell\in\Sigma}d_\ell-1
\geq|\Sigma|-1.
\]
Thus two distinct auxiliary primes always give a
nonzero negative local correction.
If every $a_\ell=b_\ell=1$, then
\[
C_\Sigma\simeq
\bigl(\Lambda/(L(T))\bigr)^{|\Sigma|-1}.
\]
The subtraction of one reflects the single diagonal
global Tate module.
Through exact control, this correction is realized at
every finite layer; through regularized measures, it
is realized by the analytic factors of
Theorem~\ref{thm:measure-main}.


\begin{thebibliography}{99}

\bibitem{DK23}
S.~Dasgupta and M.~Kakde,
On the Brumer--Stark conjecture,
\emph{Ann. of Math. (2)} \textbf{197} (2023),
289--388.

\bibitem{FW79}
B.~Ferrero and L.~C.~Washington,
The Iwasawa invariant $\mu_p$ vanishes for
abelian number fields,
\emph{Ann. of Math.} \textbf{109} (1979),
377--395.

\bibitem{Finat19}
E.~Finat,
On the growth of the $(S,\{2\})$-refined
class number,
\emph{J. Number Theory} \textbf{200} (2019),
154--160.

\bibitem{GP15}
C.~Greither and C.~D.~Popescu,
An equivariant main conjecture in Iwasawa
theory and applications,
\emph{J. Algebraic Geom.} \textbf{24} (2015),
629--692.

\bibitem{Gross88}
B.~H.~Gross,
On the values of abelian $L$-functions
at $s=0$,
\emph{J. Fac. Sci. Univ. Tokyo Sect. IA Math.}
\textbf{35} (1988), 177--197.

\bibitem{HK16}
S.~Hu and M.-S.~Kim,
The $(S,\{2\})$-Iwasawa theory,
\emph{J. Number Theory} \textbf{158} (2016),
73--89.

\bibitem{HKMS19}
S.~Hu, M.-S.~Kim, P.~Moree, and M.~Sha,
Irregular primes with respect to Genocchi
numbers and Artin's primitive root conjecture,
\emph{J. Number Theory} \textbf{205} (2019),
59--80.

\bibitem{KH12}
M.-S.~Kim and S.~Hu,
On $p$-adic Hurwitz-type Euler zeta functions,
\emph{J. Number Theory} \textbf{132} (2012),
2977--3015.

\bibitem{KW21}
H.~Knospe and L.~C.~Washington,
Dirichlet series expansions of
$p$-adic $L$-functions,
\emph{Abh. Math. Semin. Univ. Hambg.}
\textbf{91} (2021), 325--334.

\bibitem{KKMS23}
A.~Kumar, M.~Kumari, P.~Moree, and S.~K.~Singh,
Ramanujan-style congruences for prime level,
\emph{Math. Z.} \textbf{303} (2023),
Paper No.~19, 22 pp.

\bibitem{NT3}
N.~Kurokawa, M.~Kurihara, and T.~Saito,
\emph{Number Theory 3:
Iwasawa Theory and Modular Forms},
translated from the Japanese by M.~Kuwata,
Translations of Mathematical Monographs,
vol.~242,
American Mathematical Society,
Providence, RI, 2012.

\bibitem{Kur21}
M.~Kurihara,
Notes on the dual of the ideal class groups
of CM-fields,
\emph{J. Th\'eor. Nombres Bordeaux}
\textbf{33} (2021), 971--996.

\bibitem{MW84}
B.~Mazur and A.~Wiles,
Class fields of abelian extensions of $\Q$,
\emph{Invent. Math.} \textbf{76} (1984),
179--330.

\bibitem{MS23}
P.~Moree and P.~Sgobba,
Prime divisors of $\ell$-Genocchi numbers
and the ubiquity of Ramanujan-style
congruences of level $\ell$,
\emph{J. Number Theory} \textbf{251} (2023),
147--184.

\bibitem{Ribet76}
K.~A.~Ribet,
A modular construction of unramified
$p$-extensions of $\Q(\mu_p)$,
\emph{Invent. Math.} \textbf{34} (1976),
151--162.

\bibitem{Washington97}
L.~C.~Washington,
\emph{Introduction to Cyclotomic Fields},
second edition,
Graduate Texts in Mathematics, vol.~83,
Springer-Verlag, New York, 1997.

\bibitem{Weibel94}
C.~A.~Weibel,
\emph{An Introduction to Homological Algebra},
Cambridge Studies in Advanced Mathematics,
vol.~38,
Cambridge University Press, Cambridge, 1994.

\end{thebibliography}
\end{document}